\pdfoutput=1
\documentclass[11pt]{article}
\usepackage[a4paper,margin=2.5cm]{geometry}
\usepackage{amsmath,amssymb,amsthm}
\usepackage{graphicx,booktabs,array}
\usepackage[hidelinks]{hyperref}
\usepackage{microtype}
\usepackage[section]{placeins}

\newtheorem{proposition}{Proposition}[section]
\newtheorem{theorem}[proposition]{Theorem}
\newtheorem{lemma}[proposition]{Lemma}

\theoremstyle{remark}
\newtheorem{remark}[proposition]{Remark}

\newcommand{\dd}{\mathrm{d}}

\newcommand{\eps}{\varepsilon}
\renewcommand{\Re}{\operatorname{Re}}
\renewcommand{\Im}{\operatorname{Im}}

\title{Front-pinned spectral collocation for the flame problem\\ and related ignition problems}
\author{Imre Boros\\[2pt]
\small Tiberiu Popoviciu Institute of Numerical Analysis, Romanian Academy, Cluj-Napoca, Romania\\
\small \texttt{imre.boros@ictp.acad.ro}}
\date{}

\begin{document}
\maketitle

\begin{abstract}
We consider global Chebyshev collocation for the ignition problems $y'=y^q(1-y)$, $y(0)=\delta$, $q\ge2$; for $q=2$ this is the flame problem, a standard stiff test equation. After rescaling to $[-1,1]$ the solution has an interior layer of width $O(\mu)$, $\mu=(q-1)\delta^{q-1}$, at a $\delta$-dependent position. Two obstacles arise: plain Chebyshev collocation needs $O(1/\mu)$ nodes, and imposing the initial condition amplifies errors by a factor of order $1/\mu$, which makes the collocation equations severely ill-conditioned. Our main contribution addresses the second one. Since all solutions are translates of one profile, the initial condition can be replaced by an equivalent interior condition. For an unknown $u=\varphi(y)$, imposing it where the speed $|\varphi'f_q|$ of $u$ along the orbit is maximal bounds the propagator of the linearized equation by one; for $u=\ln y$ this is $y(x^\ast)=(q-1)/q$. For the first obstacle we use the sinh map of Tee and Trefethen: the time function is explicit, the complex singularities of the profile lie exactly at $\pm(2k+1)\pi i$, and the map parameters follow in closed form, so that $O(\ln(1/\mu))$ nodes suffice. In double precision the pinned scheme reaches the interpolation error of the exact solution in five or six Newton iterations for $\mu$ down to $10^{-12}$, whereas the initial-condition form fails or converges to wrong discrete solutions, also in an integral (Birkhoff) reformulation. Computations in $256$-bit arithmetic confirm the predicted rate down to errors of $10^{-70}$, AAA approximation locates the singularities to within $2\%$, and the method carries over to an Arrhenius reaction rate, with all parameters computed numerically.
\end{abstract}


\section{Introduction}

The initial value problem
\begin{equation}\label{eq:fq}
y'(t)=f_q(y):=y^q(1-y),\qquad y(0)=\delta,\qquad q\in\{2,3,\dots\},
\end{equation}
describes ignition: for small $\delta>0$ the solution stays near zero for a time of order $\delta^{1-q}$, jumps to the equilibrium $y=1$ within $O(1)$ time units, and stays there. For $q=2$ it is the flame model used to introduce stiffness in \cite{ShampineGladwellThompson2003,Moler2014}. In chemical kinetics $f_q$ is the rate of the autocatalytic step $\mathrm{A}+q\mathrm{B}\to(q+1)\mathrm{B}$ in a closed system \cite{GrayScott1990}. With $\mu:=(q-1)\delta^{q-1}$ and $x=\mu t-1$, the natural interval $t\in[0,2/\mu]$ becomes $[-1,1]$ and
\begin{equation}\label{eq:rescaled}
\mu\,\frac{\dd y}{\dd x}=y^q(1-y),\qquad y(-1)=\delta,\qquad -1\le x\le1 .
\end{equation}
This is a singularly perturbed problem \cite{OMalley1991} whose solution has an interior layer of width $O(\mu)$ near the middle of the interval. Global spectral collocation on $[-1,1]$ meets two obstacles.
\begin{itemize}
\item \emph{Resolution.} Chebyshev interpolation converges at a rate set by the distance of the nearest complex singularity from $[-1,1]$ \cite{Trefethen2013}. Here that distance is $O(\mu)$, so $O(1/\mu)$ nodes are needed. Coordinate maps can concentrate the nodes in the layer \cite{BaylissTurkel1992,TeeTrefethen2006}, but they need its position and width. Mapped Chebyshev and Laguerre collocation have been used successfully for nonlinear problems with sharp boundary layers \cite{Gheorghiu2016}; there the scaling parameters of the maps are optimised for each problem.
\item \emph{Conditioning.} The solution of \eqref{eq:fq} is extremely sensitive to its initial value: a relative change $\rho$ of $\delta$ moves the front by about $\rho/\delta^{q-1}$ time units, i.e.\ by $(q-1)\rho/\mu$ layer widths. Imposing $y(-1)=\delta$ in a collocation scheme transfers this sensitivity to the discrete nonlinear system. Difficulties of this kind are one reason why general-purpose software prefers time stepping for nonlinear initial value problems: since version 5.1, Chebfun solves them by default with \texttt{ode113} rather than by global spectral collocation, and its guide notes that the collocation approach does not converge for many nonlinear IVPs \cite{ChebfunGuide}. Section~\ref{sec:chebfun} shows this for \eqref{eq:fq}.
\end{itemize}

We remove both obstacles by exploiting the structure of \eqref{eq:fq}. All solutions are translates of one master profile, whose complex singularities we determine exactly (Section~\ref{sec:structure}) and confirm a posteriori from the computed solutions with the AAA algorithm (Section~\ref{sec:aaa}). This fixes the parameters of a sinh map \cite{TeeTrefethen2006} without any adaptivity (Section~\ref{sec:approx}); in \cite{TeeTrefethen2006} the same parameters are estimated a posteriori from the poles of Chebyshev--Pad\'e approximants. The map, and the resulting logarithmic growth of the number of nodes with $1/\mu$, are due to Tee and Trefethen; only the closed-form parameters are specific to \eqref{eq:fq}. The translation property also allows the initial condition to be replaced by an equivalent condition at the front. This turns the initial value problem into an interior-condition problem, which spectral collocation treats in the same way as a boundary value problem, and we show that the right choice of the condition bounds the propagator of the continuous linearized problem by one (Section~\ref{sec:pin}). The best condition point depends on the unknown: for $u=\varphi(y)$ it is the maximiser of $|\varphi'f_q|$ (Theorem~\ref{cor:propagator}). That the discrete problem inherits the bound is not proved here; Lemma~\ref{lem:mode} identifies the discrete mechanism, and the bound is observed in all our experiments. Section~\ref{sec:num} reports results for $q=2$ with $\delta$ down to $10^{-12}$ and for $q=3$ with $\mu$ down to $2\cdot10^{-12}$, and applies the method to an Arrhenius reaction rate, for which the time function and the singularities are only available numerically.

The contributions of the paper are:
\begin{enumerate}
\item an equivalent interior condition chosen jointly with the unknown: for $u=\varphi(y)$ the condition is imposed where $|\varphi'f|$ is maximal, which removes the amplification caused by the initial condition (Proposition~\ref{prop:pin}, Theorem~\ref{cor:propagator}, Lemma~\ref{lem:mode});
\item the exact singularity structure of the master profile (Proposition~\ref{prop:sing}), which gives the parameters of the known sinh map in closed form;
\item numerical evidence in double and $256$-bit arithmetic, including a comparison with Chebfun, an ablation that separates the effects of the map, the condition and an integral (Birkhoff) reformulation, and an Arrhenius rate.
\end{enumerate}

\paragraph{Relation to well-conditioned spectral methods.} The condition numbers of collocation matrices built from differentiation matrices grow like $N^2$, and well-conditioned alternatives replace them by pseudospectral integration matrices or ultraspherical operators \cite{WangSamsonZhao2014,Du2016,OlverTownsend2013}. These methods address the algebraic conditioning of the discretization in $N$. The difficulty here is of another kind: the solution of \eqref{eq:fq} itself depends on $\delta$ with sensitivity of order $1/\mu$ (Remark~\ref{rem:sens}), and every discretization that imposes $y(-1)=\delta$ inherits it. An integral reformulation of the same collocation equations has the same discrete solution, and by the affine invariance of Newton's method it has the same Newton iterates up to the damping and the stopping rule. In our experiments it improves the IVP form by a factor of at most about five (Section~\ref{sec:ablation}). The two approaches are complementary: front pinning can be combined with a well-conditioned discretization.

For a scalar autonomous equation the solution can also be obtained by evaluating the time function and inverting it pointwise; this is how we compute reference solutions. We use \eqref{eq:fq} as a model problem in which both obstacles can be analysed exactly: the singularities of the solution determine the map, and the linearized propagator determines where a condition can be imposed without amplification. The same two questions arise in global collocation for any initial value problem with an interior layer. There the answers have to be estimated rather than computed in closed form, and Section~\ref{sec:general} shows, for an Arrhenius reaction rate, that the quantities needed can be computed numerically when no closed form is available.

\section{Structure of the solution}\label{sec:structure}

Partial fractions give the exact time the solution needs to grow from $\delta$ to $y$:
\begin{equation}\label{eq:Theta}
T_\delta(y)=\int_\delta^y\frac{\dd z}{f_q(z)}=\Theta_q(y)-\Theta_q(\delta),\qquad \Theta_q(y)=\ln\frac{y}{1-y}-\sum_{j=1}^{q-1}\frac{1}{j\,y^j}.
\end{equation}
$\Theta_q$ is increasing from $-\infty$ to $\infty$ on $(0,1)$; let $Y_q:=\Theta_q^{-1}$.

\begin{proposition}[Translation]\label{prop:translation}
The solution of \eqref{eq:fq} is $y_\delta(t)=Y_q\bigl(t+\Theta_q(\delta)\bigr)$. Consequently the solution of \eqref{eq:rescaled} is
\begin{equation}\label{eq:yx}
y(x)=Y_q\Bigl(\frac{x-x_0}{\mu}\Bigr),\qquad x_0:=-\mu\,\Theta_q(\delta)-1 .
\end{equation}
\end{proposition}

\begin{proof}
$\Theta_q(y_\delta(t))-\Theta_q(\delta)=t$ by \eqref{eq:Theta}; and $t+\Theta_q(\delta)=(x+1)/\mu+\Theta_q(\delta)=(x-x_0)/\mu$.
\end{proof}

Here $x_0$ is the point where $s=0$, i.e.\ where $y=\Theta_q^{-1}(0)$ ($0.782$ for $q=2$, $0.862$ for $q=3$); it is not the pin point $x^\ast$ of Section~\ref{sec:pin}. For $q=2$, $Y_2(s)=1/\bigl(1+\omega(-1-s)\bigr)$ with the Wright $\omega$ function \cite{CorlessJeffrey2002}, and \eqref{eq:yx} is the Lambert-$W$ solution of \cite{Moler2014} written without overflow; $x_0=\delta\ln(1/\delta-1)$. For $q\ge3$ we evaluate $Y_q$ by bisection in the logit variable $\ln(y/(1-y))$ followed by Newton's method. The parameter $\delta$ enters \eqref{eq:yx} only through the position $x_0$ and the width $\mu$: the family is an affine rescaling of a single profile. Since the term $j=q-1$ of the sum in \eqref{eq:Theta} cancels the shift $-1$,
\[
x_0=\mu\Bigl(\ln\frac{1-\delta}{\delta}+\sum_{j=1}^{q-2}\frac{1}{j\,\delta^j}\Bigr).
\]
The front lies slightly to the right of the midpoint, at $x_0=O(\mu\ln(1/\mu))$ for $q=2$. For $q\ge3$ the term $j=q-2$ dominates and $x_0\approx(q-1)\delta/(q-2)$; this is small but corresponds to about $1/(\pi(q-2)\delta^{q-2})$ multiples of the half-width $\pi\mu$ introduced below, e.g.\ $320$ for $q=3$, $\delta=10^{-3}$.

\begin{proposition}[Complex singularities]\label{prop:sing}
Every germ of $Y_q$ can be continued analytically along every path that avoids the points $s_\ast=(2k+1)\pi i$, $k\in\mathbb Z$, and a branch that is singular at $s_\ast$ behaves like $Y_q(s)\sim\bigl(q(s-s_\ast)\bigr)^{-1/q}$. Let $\mathcal D$ be the $s$-plane cut along the rays from $\pi i$ to $+i\infty$ and from $-\pi i$ to $-i\infty$. The branch of $Y_q$ that is real on the real axis is single-valued and analytic in $\mathcal D$, and it is singular at $\pm\pi i$. Hence the solution \eqref{eq:yx} of \eqref{eq:rescaled} is analytic in the $x$-plane cut along the vertical rays from $x_0\pm i\pi\mu$ to $\pm i\infty$, and its nearest singularities are $x_0\pm i\pi\mu$.
\end{proposition}

\begin{proof}
(i) $\Theta_q'=1/f_q$ is finite and nonzero at every finite $y\notin\{0,1\}$, so a branch of $Y_q$ can be continued along a path as long as its values stay in a compact subset of $\mathbb C\setminus\{0,1\}$. At $y=0$ the function $\Theta_q$ has a pole of order $q-1$ and at $y=1$ a logarithmic singularity, so $|\Theta_q(y)|\to\infty$ as $y\to0$ or $y\to1$, on every branch. Hence, if the continuation breaks down at a finite point $s_0$ of a path, then $Y_q(s)\to\infty$ as $s\to s_0$: a finite accumulation point $y_1\notin\{0,1\}$ of $Y_q(s)$ would allow the local inverse of $\Theta_q$ near $y_1$ to continue the germ past $s_0$, and $0$ and $1$ cannot be accumulation points because $|\Theta_q|\to\infty$ there.

(ii) For $|y|>2$ we have $y/(1-y)=-1/(1-1/y)$ with $|1/y|<\tfrac12$, so $y/(1-y)$ lies in the half-plane $\Re w<0$, away from $0$. On the connected set $|y|>2$ every branch of the logarithm is therefore single-valued, and the branches, indexed by $k\in\mathbb Z$, are
\[
\ln\frac{y}{1-y}=(2k+1)\pi i-\ln\Bigl(1-\frac1y\Bigr)=(2k+1)\pi i+\sum_{m\ge1}\frac{1}{m\,y^m},
\]
so the terms $m<q$ cancel against the sum in \eqref{eq:Theta}, and $\Theta_q(y)=(2k+1)\pi i+\frac{1}{q\,y^q}+O(y^{-q-1})$. Thus $s_0$ in (i) is one of the points $(2k+1)\pi i$, these are the only finite asymptotic values of $\Theta_q$, and inverting gives the stated behaviour.

(iii) $\mathcal D$ is simply connected, contains none of these points, and contains the real axis, where the real branch is defined. By (i) and (ii) its germ can be continued along every path in $\mathcal D$, so by the monodromy theorem it defines a single-valued analytic function on $\mathcal D$.

(iv) For $y\in(1,\infty)$ approached from the upper half-plane, $\ln\frac{y}{1-y}=\ln\frac{y}{y-1}+\pi i$, so $\Theta_q(y)=\pi i+\sum_{m\ge q}\frac{1}{m\,y^m}$. Hence $\Theta_q$ maps the path from $(0,1)$ around $y=1$ through the upper half-plane to $(1,\infty)$ onto a curve in $\mathcal D$ that ends with the half-line $\{\sigma+\pi i:\sigma>0\}$, and $y\to\infty$ as $\sigma\to0^+$. So $Y_q(\sigma+\pi i)\in(1,\infty)$ tends to $\infty$ as $\sigma\to0^+$, and $\pi i$ is a singular point of the real branch; by the symmetry $Y_q(\bar s)=\overline{Y_q(s)}$, so is $-\pi i$.
\end{proof}

For $q=2$ the proposition can also be read off from the Wright $\omega$ form of $Y_2$. Since $\omega+\ln\omega=z$, the derivative $\omega'(z)=\omega/(1+\omega)$ becomes infinite where $\omega=-1$, i.e.\ at $z=-1\pm\pi i$, and these are the branch points of the principal branch of $\omega$ \cite{CorlessJeffrey2002}. There $1+\omega$ vanishes like $\sqrt{-2(z-z_\ast)}$, so $Y_2(s)=1/(1+\omega(-1-s))$ has the singularities $s=\mp\pi i$ with $Y_2\sim\bigl(2(s\pm\pi i)\bigr)^{-1/2}$. The cuts of the principal branch of $\omega$ correspond to the horizontal rays from $\pm\pi i$ to $\pm\pi i+\infty$ in the $s$-plane, a different choice of cuts than $\mathcal D$; the two branches coincide in the strip $|\Im s|<\pi$.

\begin{remark}[Location of the singularities]\label{rem:residue}
The value $\pi$ has a simple interpretation, which carries over to other right-hand sides. Let $f$ be analytic, with $f(1)=0>f'(1)$, $1/f(z)=o(z^{-1})$ as $|z|\to\infty$, and no zeros in the quarter-plane $\Re z>\tfrac12$, $\Im z\ge0$ other than $z=1$, and let $\Theta(y)=\int_{1/2}^y\dd z/f(z)$. The singular value of the inverse reached along the vertical path from $\tfrac12$ is $s_\ast=\int_{1/2}^{1/2+i\infty}\dd z/f(z)$. Deforming the vertical path to the half-line $[\tfrac12,\infty)$, indented above $z=1$, the arc at infinity contributes nothing and the indentation contributes $-\pi i\operatorname{Res}_{z=1}(1/f)=-\pi i/f'(1)$. Hence
\begin{equation}\label{eq:residue}
\Im s_\ast=\frac{\pi}{|f'(1)|},\qquad \Re s_\ast=\mathrm{p.v.}\!\int_{1/2}^{\infty}\frac{\dd z}{f(z)}.
\end{equation}
When $s_\ast$ is the nearest singularity of the real branch, as it is for $f_q$ by Proposition~\ref{prop:sing}, the half-width of the analyticity strip is thus fixed by the relaxation rate $|f'(1)|$ at the stable equilibrium; for $f_q$ it is $\pi$. For a general $f$ the contour computation alone does not exclude other singular values closer to the real axis; this needs a global argument as in Proposition~\ref{prop:sing}, or numerical evidence. For $f_q$ the other singular values, on other sheets, are the points $(2k+1)\pi i$ of Proposition~\ref{prop:sing}. If $1/f$ does not decay faster than $1/z$, the singularity can instead be an asymptotic value of $\Theta$ at another singularity of $f$; Section~\ref{sec:general} gives an example.
\end{remark}

\begin{remark}[Movable singularities]\label{rem:movable}
The equation $y'=f_q(y)$ is autonomous and has no fixed singularities: all singularities of a solution depend on its initial value, i.e.\ they are movable \cite{Hille1976}. By Propositions~\ref{prop:translation} and~\ref{prop:sing}, in the original time variable they lie at
\[
t_\ast(\delta)=-\Theta_q(\delta)\pm(2k+1)\pi i,\qquad -\Theta_q(\delta)\approx\frac1\mu .
\]
As $\delta$ decreases they move with the transition zone, parallel to the real axis. Their distance from the initial point $t=0$, which is the radius of convergence of the Taylor series of $y_\delta$ at $t=0$, grows like $1/\mu$, whereas their distance from the real axis is $\pi$ for every $\delta$. After the rescaling to $[-1,1]$ this distance becomes $\pi\mu$, which is what makes the problem hard for polynomial approximation. Since the singularities are branch points and not poles, the equation does not have the Painlevé property. We confirmed these locations in two independent ways. Padé approximants of type $[40/40]$ of the Taylor series at $t=0$, computed from the initial value alone, give $\Re t_\ast$ with a relative error below $7\cdot10^{-4}$ and $\Im t_\ast$ between $2.70$ and $3.14$ ($q=2$, $\delta=10^{-1},10^{-2},10^{-3}$; $q=3$, $\delta=10^{-1},3\cdot10^{-2}$). AAA approximation of the computed collocation solutions (Section~\ref{sec:aaa}), sampled in $t$, gives $\Re t_\ast$ to within $5\cdot10^{-3}$ and $\Im t_\ast$ between $3.18$ and $3.20$ for $\Re t_\ast$ from $12$ to $10^6$. In summary, a smaller $\delta$ does not change the singularities in the variable $s$ of the master profile, where they stay at $\pm(2k+1)\pi i$ with the same type for every $\delta$. In $t$ they move with the transition zone, parallel to the real axis and at distance $\pi$ from it; in $x$ they approach the interval, at distance $\pi\mu$; and after the sinh map of Section~\ref{sec:approx} their distance from $[-1,1]$ is about $\pi/(2L)$, $L\approx\ln(2/(\pi\mu))$, so it shrinks only logarithmically.
\end{remark}

\begin{remark}[Sensitivity]\label{rem:sens}
By Proposition~\ref{prop:translation}, replacing $\delta$ by $\delta(1+\rho)$ in the initial condition (with the equation unchanged) translates the solution by $T_\delta(\delta(1+\rho))=\rho/\bigl(\delta^{q-1}(1-\delta)\bigr)+O(\rho^2\delta^{1-q})$ time units, for small $\rho$ and $\delta$ bounded away from $1$, i.e.\ by about $(q-1)\rho/\mu$ layer widths. Any formulation that imposes $y(-1)=\delta$ inherits this amplification factor $\approx1/\mu$.
\end{remark}

\paragraph{Equilibria and phase line.} The equation $y'=f_q(y)$ has the equilibria $y=0$ and $y=1$ (Fig.~\ref{fig:phaseline}). Since $f_q'(1)=-1$, the equilibrium $y=1$ is hyperbolic and asymptotically stable \cite{Strogatz2015}; perturbations decay like $e^{-t}$, i.e.\ like $e^{-(x-x_0)/\mu}$ after the front, and this decay is the stiffness of the equation after ignition. At $y=0$ the linearization gives $f_q'(0)=0$ for $q\ge2$: the equilibrium is not hyperbolic, and its stability is decided by the leading nonlinear term, $y'\approx y^q$. Solutions with $y(0)=\delta>0$ leave it algebraically, $y(t)\approx\bigl(\delta^{1-q}-(q-1)t\bigr)^{-1/(q-1)}$, and need a time of about $1/((q-1)\delta^{q-1})=1/\mu$ to do so; for $y<0$ they approach $0$ if $q$ is even ($y=0$ is semistable) and leave it if $q$ is odd. The long slow phase and the scale $\mu$ are thus consequences of this degeneracy: for a hyperbolic unstable equilibrium ($q=1$, the logistic equation) the escape time would only be of order $\ln(1/\delta)$. Since the phase space is the line, the interval $(0,1)$ is a single heteroclinic orbit from $0$ to $1$, and every solution with $\delta\in(0,1)$ moves along it; this is the dynamical content of Proposition~\ref{prop:translation}. The conditioning of the collocation problem is governed by the linear stability along this orbit rather than at the equilibria. It is determined by the function $f_q(y)/y$, also shown in Fig.~\ref{fig:phaseline}, whose maximiser is the pin value $y^\ast$ of Section~\ref{sec:pin}.

\begin{remark}[The logistic case $q=1$]\label{rem:logistic}
The logistic equation $y'=y(1-y)$ is the member $q=1$ of the family, and it shows that both difficulties are caused by the degeneracy $f_q'(0)=0$, not by the stable equilibrium $y=1$, which is the same for all $q$. Here $\Theta_1(y)=\ln(y/(1-y))$, the master profile is $Y_1(s)=1/(1+e^{-s})$, and the first statement of Proposition~\ref{prop:translation} and Proposition~\ref{prop:sing} with its proof remain valid, with a logarithmic singularity of $\Theta_1$ at $y=0$. The singularities are again at $s=\pm(2k+1)\pi i$, in agreement with Remark~\ref{rem:residue}, but for $q=1$ they are simple poles; the logistic equation is a Riccati equation and has the Painlev\'e property. Both obstacles of Section~1 disappear.
\begin{itemize}
\item The transition is reached after $\Theta_1(\tfrac12)-\Theta_1(\delta)\approx\ln(1/\delta)$ time units, and it lasts $O(1)$ time units; the scale $\mu=(q-1)\delta^{q-1}$ degenerates. On an interval of length of order $\ln(1/\delta)$, the transition zone occupies a fraction of order $1/\ln(1/\delta)$, so it thins only logarithmically as $\delta$ decreases.
\item A relative change $\rho$ of $\delta$ moves the transition by only $T_\delta(\delta(1+\rho))\approx\rho$ time units, instead of $\rho/\delta^{q-1}$ (Remark~\ref{rem:sens}). In the variable $v=\ln y$ the propagator of Proposition~\ref{prop:pin} is governed by $Q(y)=f_1(y)/y=1-y$, which is decreasing on $[0,1]$. On the range $[\delta,1)$ of the solution it is maximal at the initial value $\delta$ (the value $(q-1)/q=0$ is the end point of the orbit at $t=-\infty$), so the initial condition is already the optimal condition, and pinning gains nothing.
\end{itemize}
\end{remark}

\begin{figure}[t]
\centering
\includegraphics[width=0.9\textwidth]{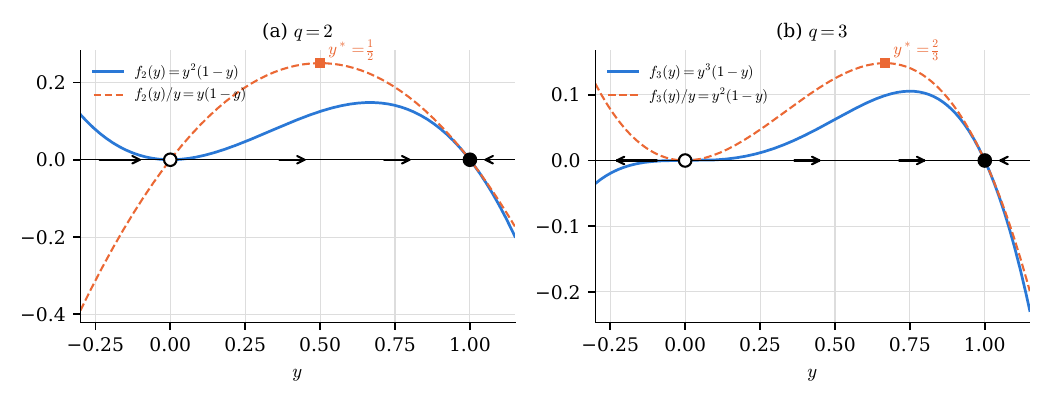}
\caption{Phase line of $y'=f_q(y)$ for $q=2,3$: the degenerate equilibrium $y=0$ (open circle), the stable equilibrium $y=1$ (filled circle), the direction of motion (arrows), and $f_q(y)/y$ (dashed), whose maximiser $y^\ast=(q-1)/q$ is the pin value.}
\label{fig:phaseline}
\end{figure}

\section{Approximability and the sinh map}\label{sec:approx}

Chebyshev interpolation of a function analytic inside the Bernstein ellipse $E_\rho$ converges like $\rho^{-N}$ \cite{Trefethen2013}. By Proposition~\ref{prop:sing}, the nearest singularities of $y$ lie at distance $\pi\mu$ from the real axis close to the midpoint, so $\rho\approx1+\pi\mu$ and the error behaves like $e^{-\pi\mu N}$. Accuracy $10^{-10}$ requires $N\approx23/(\pi\mu)$ nodes: $733$ for $\mu=10^{-2}$ and about $7300$ for $\mu=10^{-3}$.

Because the location and the distance of the singularities are known exactly, we use the sinh map of Tee and Trefethen \cite{TeeTrefethen2006},
\begin{equation}\label{eq:sinh}
x=g(\xi)=x_0+\eps\sinh(L\xi+M),\qquad L=\tfrac12(s_+-s_-),\ M=\tfrac12(s_++s_-),\ s_\pm=\operatorname{arcsinh}\frac{\pm1-x_0}{\eps},
\end{equation}
with $x_0$ from \eqref{eq:yx} and $\eps=\pi\mu$, and collocate at the images of the $N+1$ Chebyshev--Lobatto points $\xi_j$, using the standard differentiation matrix \cite{Trefethen2000} scaled by $1/g'(\xi_j)$. In the variable $s=(x-x_0)/\mu$ of the master profile the nodes are $s_j=\pi\sinh(L\xi_j+M)$, which we compute directly, so that the exact solution can be evaluated without cancellation. The singularities $x_0\pm i\eps$ are mapped to $\xi=(-M\pm i\pi/2)/L$, at distance $\pi/(2L)$ from $[-1,1]$. More precisely, $g$ maps the strip $|\Im\xi|<\pi/(2L)$ conformally onto the $x$-plane cut along the vertical rays from $x_0\pm i\eps$ to $\pm i\infty$. By Proposition~\ref{prop:sing}, $y$ is analytic in this cut plane, so $y\circ g$ is analytic in the whole strip, and its singularities nearest to $[-1,1]$ are $z_\ast,\bar z_\ast$ with $z_\ast=(-M+i\pi/2)/L$. The scheme of Section~\ref{sec:pin} interpolates $v=\ln y$ rather than $y$. By the proof of Proposition~\ref{prop:sing}, $Y_q(s)\to0$ only as $|s|\to\infty$, so $Y_q$ has no zeros, and $\ln Y_q$ is analytic in every simply connected domain in which $Y_q$ is, in particular in the strip $|\Im s|<\pi$ and in the cut plane above. At $s_\ast$ it has the logarithmic branch points $\ln Y_q(s)\sim-\frac1q\ln\bigl(q(s-s_\ast)\bigr)$. The change of unknown changes the type of the singularities but not their location, and the following applies to $v\circ g$ as well as to $y\circ g$. Chebyshev interpolation in $\xi$ therefore converges like $\rho_\ast^{-N}$, $\rho_\ast=|z_\ast+(z_\ast^2-1)^{1/2}|>1$ (with the branch of the square root that gives the larger modulus), up to algebraic factors that depend on the type of the singularity. For $M=0$, $\ln\rho_\ast=\operatorname{arcsinh}(\pi/(2L))$; since $|M|$ is small and $L$ large, $\ln\rho_\ast\approx\pi/(2L)$. For $q=2$, $\delta=10^{-2}$, $\ln\rho_\ast=0.3698$ and $\pi/(2L)=0.3783$, and for $\mu\le2\cdot10^{-4}$ the two differ by less than $0.7\%$. The predicted convergence rate is thus
\begin{equation}\label{eq:rate}
\text{error}\approx C\,e^{-\pi N/(2L)},\qquad L\approx\ln\frac{2}{\pi\mu},
\end{equation}
and the number of nodes for a fixed accuracy is predicted to grow only like $\ln(1/\mu)$; for accuracy $10^{-10}$, $N\approx\frac{2L}{\pi}\ln10^{10}$. Fig.~\ref{fig:grid}(a) shows how the nodes resolve the front for $q=2$, $\delta=10^{-6}$ and $N=128$.

The map must be centred correctly: the sinh map concentrates the nodes within a few multiples of $\eps$ of $x_0$. For $q=2$, centring it at the midpoint $x_0=0$ instead of the exact $x_0=\delta\ln(1/\delta-1)$ is an offset of only $2.9$, $4.4$ and $5.9$ multiples of $\eps$ for $\delta=10^{-4},10^{-6},10^{-8}$. Even so, it degrades the interpolation error at $N=512$ from machine precision to $1.2\cdot10^{-9}$, $5.3\cdot10^{-5}$ and $2.3\cdot10^{-3}$. For $q\ge3$ the offset is hundreds or thousands of multiples of $\eps$ (Section~\ref{sec:structure}), and a map centred at the midpoint does not resolve the front at all.

\begin{figure}[t]
\centering
\includegraphics[width=\textwidth]{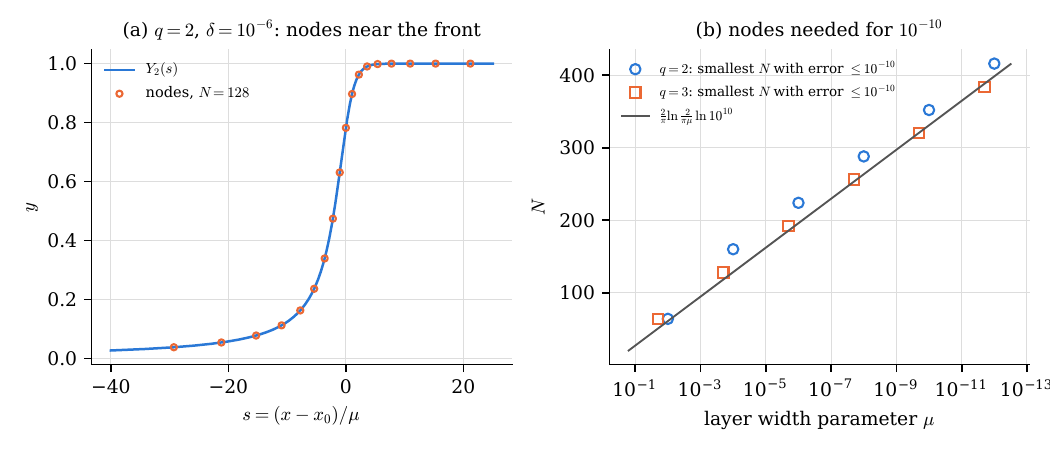}
\caption{(a) Master profile $Y_2$ and the pinned collocation solution at the sinh-mapped nodes lying in $s\in[-40,25]$ ($q=2$, $\delta=10^{-6}$, $N=128$; $s=(x-x_0)/\mu$, $\mu=\delta$). (b) Smallest $N\in\{32,64,\dots,512\}$ for which the pinned solution has maximal error $\le10^{-10}$, versus the layer width $\mu$ ($\mu=\delta$ for $q=2$, $\mu=2\delta^2$ for $q=3$); line: the prediction from \eqref{eq:rate}.}
\label{fig:grid}
\end{figure}

\section{Initial condition versus front pinning}\label{sec:pin}

We solve for $v=\ln y$, which satisfies
\begin{equation}\label{eq:logvar}
\mu\,v'(x)=e^{(q-1)v}-e^{qv}.
\end{equation}
In the slow phase $y$ varies over many orders of magnitude while $v$ varies only logarithmically, and relative errors in $y$ become absolute errors in $v$. The phase line of \eqref{eq:logvar} keeps only one of the two equilibria of Fig.~\ref{fig:phaseline}. The right-hand side $F_q(v)=e^{(q-1)v}-e^{qv}$ vanishes only at $v=0$, with $F_q'(0)=-1$, so the stable equilibrium $y=1$ keeps its rate. The degenerate equilibrium $y=0$ is moved to $v=-\infty$, and the slow phase becomes a drift with speed $F_q(v)\approx e^{(q-1)v}$; it takes $\int_{\ln\delta}^{0}e^{-(q-1)v}\,\dd v\approx1/\mu$ time units, as in Section~\ref{sec:structure}. The collocation equations are solved by damped Newton iteration, as is usual for nonlinear spectral collocation \cite{Boyd2001}: the step length is halved, down to $10^{-3}$, until the Euclidean norm of the residual decreases. The initial guess is the generic profile
\[
y^{(0)}=\frac{1}{1+\bigl((q-1)\,\mathrm{softplus}(s_0-s)\bigr)^{1/(q-1)}},\qquad \mathrm{softplus}(z)=\ln(1+e^{z}),
\]
in the variable $s=(x-x_0)/\mu$ of \eqref{eq:yx}. It has the correct algebraic left tail $y\sim(-(q-1)s)^{-1/(q-1)}$, saturates at $1$, and $s_0$ is chosen so that it equals the pin value $y^\ast$ below at the pin point $s=\Theta_q(y^\ast)$. Two formulations are compared.
\begin{itemize}
\item \emph{IVP form}: \eqref{eq:logvar} collocated at $\xi_0,\dots,\xi_{N-1}$, plus $v(-1)=\ln\delta$.
\item \emph{Pinned form}: \eqref{eq:logvar} collocated at the same points, plus the interior condition $v(x^\ast)=\ln y^\ast$, imposed through the barycentric interpolation row \cite{BerrutTrefethen2004} at $\xi^\ast=g^{-1}(x^\ast)$. Here $y^\ast=(q-1)/q$ and $x^\ast=x_0+\mu\,\Theta_q(y^\ast)$.
\end{itemize}
By Proposition~\ref{prop:translation} both conditions select the same solution of the differential equation, and $x^\ast$ is known in closed form. Their conditioning is very different. Interior conditions that remove a translation invariance are standard in the computation of travelling waves and connecting orbits, where a phase condition fixes the otherwise free shift \cite{Beyn1990,BeynThummler2004}. Here the problem has no free shift; the interior condition replaces a given initial condition by an equivalent one, and its position is chosen to make the linearized problem well conditioned.

\begin{proposition}\label{prop:pin}
Let $Q(t)=f_q(y_\delta(t))/y_\delta(t)=y_\delta^{q-1}(1-y_\delta)$. The linearization of $v'=e^{(q-1)v}-e^{qv}$, written in the time variable $t$ ($\mu\,\dd/\dd x=\dd/\dd t$), about $v=\ln y_\delta$ is $w'=\bigl((q-1)y^{q-1}-qy^q\bigr)w+r$, whose propagator, the pointwise amplification factor from $s$ to $t$, is $P(t,s)=Q(t)/Q(s)$.
\begin{enumerate}
\item For the IVP form, $w(t)=\int_0^tP(t,s)r(s)\,\dd s$. At the front, $P(t^\ast,0)=Q(y^\ast)/Q(\delta)\approx\frac{1}{q}\bigl(\frac{q-1}{q}\bigr)^{q-1}\delta^{1-q}$, which equals $1/(4\delta)$ for $q=2$.
\item For the pinned form with the pin value perturbed by $\eta$, $w(t)=P(t,t^\ast)\eta+\int_{t^\ast}^{t}P(t,s)r(s)\,\dd s$ with $|P(t,s)|\le1$ for all $s$ between $t^\ast$ and $t$, so $|w(t)|\le|\eta|+\bigl|\int_{t^\ast}^t|r|\bigr|$.
\end{enumerate}
\end{proposition}

\begin{proof}
$\frac{\dd}{\dd t}\ln Q=(q-1)\frac{y'}{y}-\frac{y'}{1-y}=(q-1)y^{q-1}(1-y)-y^q=(q-1)y^{q-1}-qy^q$, which is the coefficient of the linearization. Since $y\mapsto y^{q-1}(1-y)$ attains its maximum on $[0,1]$ at $y^\ast=(q-1)/q$, $Q$ increases up to $t^\ast$ and decreases afterwards. Hence $Q(t)\le Q(s)$ whenever $s$ lies between $t^\ast$ and $t$.
\end{proof}

Proposition~\ref{prop:pin} concerns the continuous linearized problem. It does not bound the inverse of the collocation matrix uniformly in $\mu$, and for $r$ measured in the maximum norm the integral still grows with the length of the interval, which is of order $1/\mu$ in $t$. What it shows is that no perturbation is amplified on its way to the transition zone. That the discrete problem behaves in the same way is observed in Section~\ref{sec:num}, where the condition numbers of the pinned form after row scaling are bounded independently of $\delta$; it is not proved here.

Since $Q=f_q(y_\delta)/y_\delta=v'$, the propagator is the ratio of two values of $v'$. The same holds for every choice of the unknown.

\begin{theorem}[Choice of the condition point]\label{cor:propagator}
Let $f$ be $C^1$ on an open interval $I$ with $f>0$ on $I$, let $y$ solve $y'=f(y)$ with values in $I$ on $[0,T]$ (so $y$ is increasing), and let $\varphi$ be $C^2$ on $I$ with $\varphi'\ne0$. Then $u=\varphi(y)$ satisfies $u'=F(u)$ with $F=(\varphi'f)\circ\varphi^{-1}$, and the linearization $w'=F'(u(t))\,w+r$ has the propagator
\[
P(t,s)=\frac{u'(t)}{u'(s)}=\frac{h(y(t))}{h(y(s))},\qquad h=\varphi'f .
\]
Assume that $|h|$ is nondecreasing on $y([0,T])$ up to a point $y_\varphi$ and nonincreasing after it ($y_\varphi$ may be an end point; if $|h|$ is maximal on an interval, any point of it can be taken).
\begin{enumerate}
\item[(a)] For the condition imposed at $t_\varphi$ with $y(t_\varphi)=y_\varphi$, $|P(t,s)|\le1$ for all $s$ between $t_\varphi$ and $t$.
\item[(b)] For the condition imposed instead at $t_p$ with $|h(y(t_p))|\ge\theta\max_{z\in y([0,T])}|h(z)|$, $0<\theta\le1$, $|P(t,s)|\le1/\theta$ for all $s$ between $t_p$ and $t$.
\end{enumerate}
\end{theorem}

\begin{proof}
Differentiating $u'=F(u)$ gives $(u')'=F'(u)\,u'$, so $u'$ solves the homogeneous linearized equation, and $u'=h(y)\ne0$ because $f\ne0$ on $I$. In the scalar case every solution is a multiple of $u'$, which gives $P$. Since $y$ is increasing, $y(s)$ lies between $y(t_p)$ and $y(t)$, and by the monotonicity of $|h|$ on either side of $y_\varphi$, $|h(y(s))|\ge\min\{|h(y(t_p))|,|h(y(t))|\}$. Hence $|P(t,s)|\le\max\{1,|h(y(t))|/|h(y(t_p))|\}\le\max\{1,1/\theta\}$, which is (b), and (a) is the case $\theta=1$.
\end{proof}

Proposition~\ref{prop:pin} is the case $f=f_q$, $I=(0,1)$, $\varphi=\ln$. By (b), a condition at a point where $Q$ is a fraction $\theta$ of its maximum amplifies by at most $1/\theta$, so the pin need not be placed exactly; Section~\ref{sec:ablation} confirms this. For the unknowns compared below:
\begin{center}
\begin{tabular}{llll}
\hline
unknown $u$ & $h=u'$ & best condition point $y_\varphi$ & range of $u$\\
\hline
$y$ & $y^q(1-y)$ & $q/(q+1)$ & $[\delta,1]$\\
$v=\ln y$ & $y^{q-1}(1-y)$ & $(q-1)/q$ & $[\ln\delta,0]$\\
$z=y^{1-q}$ & $(1-q)(1-y)$ & $\delta$ (initial point) & $[1,\delta^{1-q}]$\\
$\ln\frac{y}{1-y}$ & $y^{q-1}$ & right end & about $[\ln\delta,1/\mu]$\\
\hline
\end{tabular}
\end{center}
For $z=y^{1-q}$ the initial condition is itself optimal. For the logit the optimal point is the right end, where $y=1$ to working precision, so that the logit cannot be recovered from $y$ there, and its range grows like $1/\mu$; we did not test it. The amplification factor of the IVP form is the sensitivity of Remark~\ref{rem:sens}: discretization and rounding errors near $x=-1$ become shifts of the front. Pinning does not remove this sensitivity from the problem; it moves it into the position of the interval ends relative to the front, $s_\pm=(\pm1-x_0)/\mu$, with $s_-=\Theta_q(\delta)$. Computed from the closed form, this is harmless: an error $e$ in $s_-$ changes the solution at $x=-1$ only by $f_q(\delta)\,e$, i.e.\ relatively by about $\mu e$. The pinned form also gives the initial value as an output, and $|y(-1)/\delta-1|$ is a useful a posteriori check.

The mechanism behind the amplification is visible in the discrete problem as well. We number the Chebyshev--Lobatto points from $\xi_0=1$ to $\xi_N=-1$, so that \eqref{eq:logvar} is collocated at $\xi_0,\dots,\xi_{N-1}$ and the last row holds the condition. Let $\phi_j=v'(x_j)$ be the nodal values of the derivative of the exact solution, i.e.\ of the tangent to the family of its translates, and let $e_N$ be the unit vector of the condition row.

\begin{lemma}[Discrete translation mode]\label{lem:mode}
Let $J$ be the Newton matrix of the IVP or the pinned form, evaluated at the nodal values of the exact solution $v$ of \eqref{eq:logvar}, and $D_x$ the differentiation matrix in $x$. Then $J\phi=c\,e_N+\rho$ with $\rho_j=\mu\bigl((D_x\phi)_j-v''(x_j)\bigr)$ for $j<N$, and
\begin{itemize}
\item for the IVP form, $c=v'(-1)$ and $\rho_N=0$;
\item for the pinned form, $c=v'(x^\ast)$ and $\rho_N=(I_Nv')(x^\ast)-v'(x^\ast)$, with $I_N$ the interpolant on the grid.
\end{itemize}
Consequently $J^{-1}e_N=(\phi-J^{-1}\rho)/c$.
\end{lemma}

\begin{proof}
For $j<N$, $(J\phi)_j=\mu(D_x\phi)_j-F_q'(v(x_j))\,\phi_j$, and differentiating $\mu v'=F_q(v)$ gives $\mu v''=F_q'(v)\,v'$, so $(J\phi)_j=\rho_j$. The condition row is $e_N^T$ for the IVP form and the barycentric interpolation row at $\xi^\ast$ for the pinned form, which gives the last component.
\end{proof}

The column $J^{-1}e_N$ is the response of the discrete solution to a perturbation of the condition datum, and up to the term $J^{-1}\rho/c$, which involves the error of differentiating the interpolant of $v'$, it is the translation mode $\phi/c$. For the IVP form $c=(1-\delta)/(q-1)$ and $\|\phi\|_\infty=Q(y^\ast)/\mu$, so the maximum of this column is about $Q(y^\ast)/\delta^{q-1}$, the amplification factor of Proposition~\ref{prop:pin}. For the pinned form $c=\|\phi\|_\infty$, and the maximum is about $1$. Computed at $N=N_{\min}$ of Tables~\ref{tab:colloc} and~\ref{tab:q3colloc}, with $J$ evaluated at the converged pinned solution, the maximum for the IVP form agrees with $Q(y^\ast)/(\delta^{q-1}(1-\delta))$ to within $2\%$ as long as the $2$-norm condition number of $J$ is below about $3\cdot10^{15}$ ($q=2$, $\delta\ge10^{-8}$; $q=3$, $\delta\ge10^{-4}$); for smaller $\delta$ the IVP matrix is numerically singular. For the pinned form it lies between $0.994$ and $1.000$ for every $\delta$. The lemma does not give a stability bound, because $J^{-1}\rho$ is not controlled uniformly in $\mu$; it shows that the condition is best imposed where the unknown changes fastest, which is Theorem~\ref{cor:propagator} in discrete form.

\begin{remark}[Choice of the unknown]\label{rem:variable}
Theorem~\ref{cor:propagator} decides where the condition can be imposed. A second property of the unknown is its dynamic range, since a global differentiation matrix couples all nodes, so that rounding errors made in the largest values reach the transition zone. Table~\ref{tab:variables} compares, for $q=2$ and on the grids of Table~\ref{tab:colloc}, the unknowns $y$, $v=\ln y$ and $z=1/y$, each with the initial condition and with an interior condition.
\begin{itemize}
\item For $y$ itself the initial condition already loses accuracy for $\delta=10^{-2}$; from $\delta=10^{-4}$ on Newton's method does not converge or converges to a wrong discrete solution. The optimal interior condition, at $y=2/3$, works down to $\delta=10^{-8}$, but for $\delta\le10^{-10}$ Newton's method does not converge in $80$ iterations.
\item For $z=1/y$ the initial condition is optimal, and the interior condition at $y=1/2$ amplifies by at most $2$: Newton's method converges in two to five iterations in both forms and for every $\delta$. However, $z$ ranges from $1$ to $1/\delta$, and for $\delta\le10^{-4}$ the error grows like $u/\delta$, with $u$ the unit roundoff, up to $2\cdot10^{-4}$ for $\delta=10^{-12}$.
\item Only $v=\ln y$ combines a propagator bounded by one (with the pin) with a moderate range, $\ln\delta\le v\le0$. Its pinned form reaches the interpolation error for every $\delta$.
\end{itemize}
\end{remark}

\paragraph{Algorithm.} The pinned scheme for \eqref{eq:fq} consists of the following steps.
\begin{enumerate}
\item Compute $\mu$, $x_0=-\mu\Theta_q(\delta)-1$, the pin value $y^\ast$ (the maximiser of $|h|$) and its position $s_{\mathrm{pin}}=\Theta_q(y^\ast)$; for a general $f$ by quadrature (Section~\ref{sec:general}).
\item Take a nearest singularity $s_{\mathrm{sing}}$ in the upper half-plane (here $\pi i$; for a general $f$ from Remark~\ref{rem:residue}, quadrature or AAA) and set $\eps=\mu\Im s_{\mathrm{sing}}$, centre $x_0+\mu\Re s_{\mathrm{sing}}$.
\item Form the nodes $x_j=g(\xi_j)$ from $\xi_0=1$ to $\xi_N=-1$, compute $s_j$ directly from \eqref{eq:sinh}, and $D_x=\mathrm{diag}(1/g'(\xi_j))\,D$.
\item Residual: $R_j=\mu(D_xv)_j-F_q(v_j)$ for $j<N$, $R_N=\ell^Tv-\ln y^\ast$ with the barycentric row $\ell$ at $\xi^\ast=g^{-1}(x^\ast)$, $x^\ast=x_0+\mu s_{\mathrm{pin}}$. Jacobian: $J=\mu D_x-\mathrm{diag}(F_q'(v))$ with last row $\ell^T$.
\item Start from the generic profile, with the softplus function evaluated as $z$ for $z>30$ to avoid overflow.
\item Damped Newton: $\Delta=-J^{-1}R$. If $\max|\Delta|<10^{-12}$ and $\|R(v+\Delta)\|_\infty<10^{-10}$, accept $v+\Delta$ and stop. Otherwise take the first $\lambda\in\{1,\tfrac12,\dots,2^{-10}\}$ with $\|R(v+\lambda\Delta)\|_2<(1-10^{-4}\lambda)\|R(v)\|_2$, and $\lambda=2^{-10}$ if there is none. Declare failure after $80$ iterations.
\item Check the recovered initial value $|y(-1)/\delta-1|$, the final nodal residual and the off-grid defect (Section~\ref{sec:q2}).
\end{enumerate}

\section{Numerical results}\label{sec:num}

All computations were done with the scripts in the public repository listed in the Declarations; the AI assistant Claude (Anthropic) was used to help write and check this code and to verify the reported numbers against the saved outputs, and all results were checked by the author. The errors below are sampled maximum errors, evaluated on $4000$ Chebyshev points of the first kind in $\xi$, of the barycentric interpolant of the nodal values $y_j=e^{v_j}$, measured against the exact solution \eqref{eq:yx}. Exponentiating the interpolant of the $v_j$ instead gives errors $4$ to $20$ times smaller at $N=N_{\min}$ in Tables~\ref{tab:colloc} and~\ref{tab:q3colloc}, in line with the milder, logarithmic singularities of $\ln Y_q$ (Section~\ref{sec:approx}); we report the more conservative values. Throughout, $N$ is the polynomial degree, and the grid has $N+1$ points. We use the sixteen values $N=32,64,\dots,512$ and stop Newton's method as in step~6 of the algorithm; for the unknown $1/y$ in Table~\ref{tab:variables} both tolerances are relative to $\max|u|$. A run that stops in this way is called converged; a converged run with error above $10^{-2}$ is called a wrong discrete solution. Runs in which the residual stagnates at the level of rounding errors while the full correction $\Delta$ does not become small count as not converged. In all $645$ runs of Tables~\ref{tab:colloc}--\ref{tab:q3colloc} and~\ref{tab:arrhenius} and Figs.~\ref{fig:colloc}, \ref{fig:q3} and~\ref{fig:general} that are classified as converged, the final residual is at most $5.0\cdot10^{-12}$ (for the integral form of Section~\ref{sec:ablation} measured in the differentiation form). A pinned solve with $N=384$ ($q=2$, $\delta=10^{-10}$, five Newton steps with dense LU factorization) takes about $0.03$ seconds (Python 3.13, NumPy 2.5.3 with OpenBLAS 0.3.34 on two threads of an Intel Xeon processor at 2.1\,GHz, Linux).

\subsection{\texorpdfstring{The flame problem, $q=2$}{The flame problem, q=2}}\label{sec:q2}

Table~\ref{tab:colloc} and Fig.~\ref{fig:colloc} show the results for $\delta=\mu$ from $10^{-2}$ to $10^{-12}$.
\begin{itemize}
\item The pinned form converges in exactly five Newton iterations for every $\delta$ and every $N$. Its error is $0.85$--$1.01$ times the interpolation error of the exact solution on the same grid, down to machine precision.
\item The initial value is recovered as an output: $|y(-1)/\delta-1|\le6\cdot10^{-13}$ for $\delta\le10^{-4}$ and $N\ge N_{\min}$.
\item In the IVP form, Newton needs more iterations. It does not converge within $80$ iterations for $\delta\le10^{-8}$ and $N\ge192$, for $\delta=10^{-6}$ and $N\ge128$, and in eleven further cases for $\delta=10^{-5}$ and $10^{-4}$; it also converges to wrong discrete solutions, with errors near $0.1$ (e.g.\ $\delta=10^{-8}$, $N=128$ and $\delta=10^{-10}$, $N=160$). Where it converges to the right solution, its accuracy is limited to about $10u/\delta$, with $u$ the unit roundoff.
\item The $2$-norm condition numbers of the final Newton matrices at $N=256$--$512$ grow like $\delta^{-2}$ in the IVP form ($3.3\cdot10^{11}$--$7.0\cdot10^{11}$ at $\delta=10^{-6}$), up to $10^{16}$--$5\cdot10^{18}$ for $\delta\le10^{-10}$. In the pinned form they grow only like $\delta^{-1}$ ($2.0\cdot10^{6}$--$4.2\cdot10^{6}$ and $1.1\cdot10^{10}$--$2.4\cdot10^{10}$ at $\delta=10^{-6}$ and $10^{-10}$); this growth comes from the scaling of the matrix rows and does not affect the accuracy. After each row is scaled to unit $\infty$-norm, the condition numbers of the pinned form lie between $93$ and $499$ for every $\delta$ ($95$ and $369$ for $q=3$, $\delta=10^{-2},10^{-4},10^{-6}$), whereas those of the IVP form still grow, from about $10^4$ at $\delta=10^{-2}$ to $10^7$--$10^{12}$ for $\delta\le10^{-6}$.
\item The outcome of the IVP form depends on rounding errors. Repeating all $144$ runs (nine values of $\delta$ from $10^{-1}$ to $10^{-12}$, sixteen of $N$) with one instead of two BLAS threads, which only changes the order of the floating-point operations, changed the number of Newton iterations of the IVP form in $36$ runs, its error by more than $0.1\%$ in $85$, and the convergence or non-convergence of Newton's method in $8$. The errors of the pinned form changed by at most $5\cdot10^{-4}$ relatively, and its iteration counts not at all. The iteration counts of the IVP form in Tables~\ref{tab:colloc} and~\ref{tab:q3colloc} are therefore indicative only.
\item A check that needs no exact solution is the off-grid defect $\max|\mu p'-F_q(p)|$ of the interpolant $p$ of the computed $v$ on the $4000$ test points, with $F_q(v)=e^{(q-1)v}-e^{qv}$. For the pinned form at $N_{\min}$ it is at most $5\cdot10^{-11}$, for $q=2$ and $q=3$ (Tables~\ref{tab:colloc} and~\ref{tab:q3colloc}). For the IVP solutions with errors of $2\cdot10^{-3}$ or more (at $N_{\min}$ and $N=384$) it is $6.9\cdot10^{-10}$--$5.8\cdot10^{-4}$, so it exposes wrong discrete solutions. It does not expose phase errors: for $q=2$, $\delta=10^{-6}$ and $10^{-8}$ at $N_{\min}$, the IVP errors are $9.4\cdot10^{-9}$ and $2.4\cdot10^{-6}$, but the defect is the same as for the pinned form, because a translate of the solution satisfies the differential equation. Phase errors are detected by the recovered initial value of the pinned form.
\item Unmapped Chebyshev collocation, even in the pinned form, needs $768$ nodes for $10^{-10}$ at $\delta=10^{-2}$ (prediction $733$) and has error $1.4\cdot10^{-3}$ with $1536$ nodes at $\delta=10^{-3}$ (Fig.~\ref{fig:colloc}b).
\end{itemize}

\begin{table}[t]
\centering\small
\caption{$q=2$: sinh-mapped Chebyshev collocation for $\delta\,y'=y^2(1-y)$, $y(-1)=\delta$. $N_{\min}$: smallest $N$ for which the pinned solution has maximal error $\le10^{-10}$; error, Newton iterations and recovered initial value at $N_{\min}$; last two columns: IVP form at $N=384$ ($^\ast$: Newton did not converge in 80 iterations).}
\label{tab:colloc}
\begin{tabular}{ccccccccc}
\toprule
 & & & \multicolumn{4}{c}{pinned form} & \multicolumn{2}{c}{IVP form, $N=384$}\\
\cmidrule(lr){4-7}\cmidrule(lr){8-9}
$\delta$ & $L$ & $\frac{2L}{\pi}\ln10^{10}$ & $N_{\min}$ & error & its & $|y(-1)/\delta-1|$ & error & its\\
\midrule
$10^{-2}$ & 4.15 & 61 & 64 & $6.8\cdot10^{-11}$ & 5 & $5.5\cdot10^{-12}$ & $1.3\cdot10^{-13}$ & 6 \\
$10^{-4}$ & 8.76 & 128 & 160 & $4.7\cdot10^{-13}$ & 5 & $2.7\cdot10^{-15}$ & $7.5\cdot10^{-12}$ & 21 \\
$10^{-6}$ & 13.36 & 196 & 224 & $4.6\cdot10^{-12}$ & 5 & $7.0\cdot10^{-14}$ & $5.2\cdot10^{-9}{}^{\ast}$ & 80 \\
$10^{-8}$ & 17.97 & 263 & 288 & $1.4\cdot10^{-11}$ & 5 & $1.4\cdot10^{-13}$ & $1.9\cdot10^{-8}{}^{\ast}$ & 80 \\
$10^{-10}$ & 22.57 & 331 & 352 & $2.8\cdot10^{-11}$ & 5 & $8.2\cdot10^{-14}$ & $8.6\cdot10^{-2}{}^{\ast}$ & 80 \\
$10^{-12}$ & 27.18 & 398 & 416 & $4.4\cdot10^{-11}$ & 5 & $5.3\cdot10^{-13}$ & $9.8\cdot10^{-2}{}^{\ast}$ & 80 \\
 
\bottomrule
\end{tabular}
\end{table}

\begin{figure}[t]
\centering
\includegraphics[width=\textwidth]{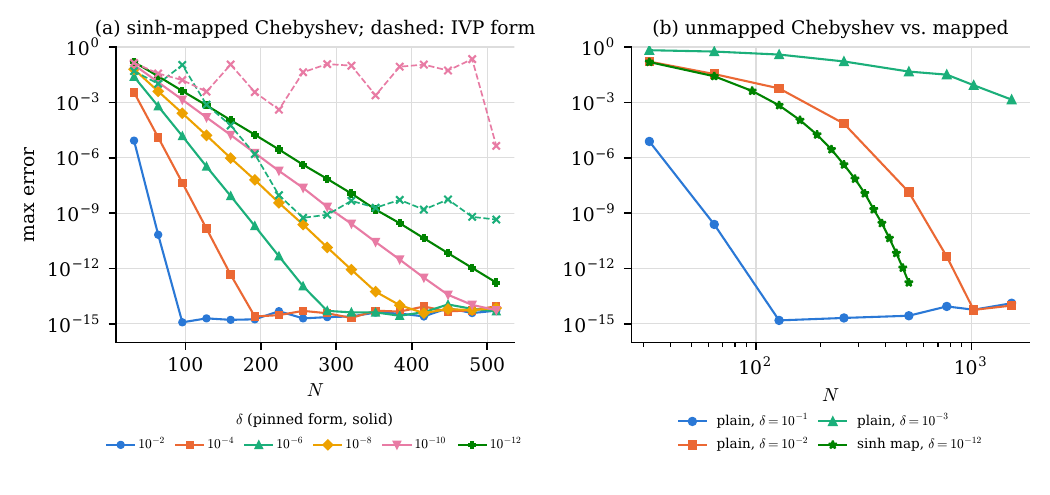}
\caption{$q=2$. (a) Maximal error of the sinh-mapped collocation versus $N$ (solid: pinned form; dashed: IVP form for $\delta=10^{-6},10^{-10}$). (b) Unmapped Chebyshev collocation (pinned) for $\delta=10^{-1},10^{-2},10^{-3}$ compared with the mapped method at $\delta=10^{-12}$.}
\label{fig:colloc}
\end{figure}

\begin{table}[t]
\centering\footnotesize
\setlength{\tabcolsep}{3pt}
\caption{$q=2$: maximal error of the sinh-mapped collocation solution for three choices of the unknown, with the initial condition (IVP) and with an interior condition (pin at $y=2/3$ for $y$, at $y=1/2$ for $\ln y$ and $1/y$), at $N=N_{\min}$ of Table~\ref{tab:colloc}. Third column: interpolation error of the exact solution on the same grid; $^\ast$: Newton's method did not converge in $80$ iterations.}
\label{tab:variables}
\begin{tabular}{ccccccccc}
\toprule
 & & & \multicolumn{2}{c}{$y$} & \multicolumn{2}{c}{$v=\ln y$} & \multicolumn{2}{c}{$z=1/y$}\\
\cmidrule(lr){4-5}\cmidrule(lr){6-7}\cmidrule(lr){8-9}
$\delta$ & $N$ & interpolation & IVP & pinned & IVP & pinned & IVP & pinned\\
\midrule
$10^{-2}$ & 64 & $7.1\cdot10^{-11}$ & $9.6\cdot10^{-8}$ & $7.3\cdot10^{-11}$ & $1.5\cdot10^{-10}$ & $6.8\cdot10^{-11}$ & $7.1\cdot10^{-11}$ & $7.1\cdot10^{-11}$ \\
$10^{-4}$ & 160 & $4.9\cdot10^{-13}$ & $4.1\cdot10^{-6}{}^{\ast}$ & $3.0\cdot10^{-13}$ & $1.9\cdot10^{-12}$ & $4.7\cdot10^{-13}$ & $2.8\cdot10^{-12}$ & $6.6\cdot10^{-13}$ \\
$10^{-6}$ & 224 & $4.7\cdot10^{-12}$ & $9.3\cdot10^{-2}{}^{\ast}$ & $2.9\cdot10^{-12}$ & $9.4\cdot10^{-9}{}^{\ast}$ & $4.6\cdot10^{-12}$ & $2.6\cdot10^{-10}$ & $5.1\cdot10^{-11}$ \\
$10^{-8}$ & 288 & $1.4\cdot10^{-11}$ & $1.0\cdot10^{-1}$ & $9.4\cdot10^{-12}$ & $2.4\cdot10^{-6}{}^{\ast}$ & $1.4\cdot10^{-11}$ & $1.2\cdot10^{-8}$ & $3.0\cdot10^{-9}$ \\
$10^{-10}$ & 352 & $2.8\cdot10^{-11}$ & $3.9\cdot10^{-2}{}^{\ast}$ & $1.6\cdot10^{-6}{}^{\ast}$ & $2.4\cdot10^{-3}{}^{\ast}$ & $2.8\cdot10^{-11}$ & $8.4\cdot10^{-7}$ & $1.7\cdot10^{-7}$ \\
$10^{-12}$ & 416 & $4.4\cdot10^{-11}$ & $4.0\cdot10^{-2}{}^{\ast}$ & $5.5\cdot10^{-7}{}^{\ast}$ & $1.7\cdot10^{-1}{}^{\ast}$ & $4.4\cdot10^{-11}$ & $2.1\cdot10^{-4}$ & $1.7\cdot10^{-5}$ \\
 
\bottomrule
\end{tabular}
\end{table}

\subsection{Ablation, integral reformulation and robustness}\label{sec:ablation}

Table~\ref{tab:ablation} separates the three ingredients for $q=2$ at $N=512$: the map, the condition, and the formulation of the collocation equations. Besides the differentiation form used so far we use an integral (Birkhoff) form in the spirit of \cite{WangSamsonZhao2014,Du2016}: the unknowns are $\dd v/\dd\xi$ at $\xi_0,\dots,\xi_{N-1}$ and $v(-1)$, and $v=v(-1)+B\,\dd v/\dd\xi$ with the pseudospectral integration matrix $B$ of the Lagrange basis on $\xi_0,\dots,\xi_{N-1}$. Both forms define the same discrete solution.
\begin{itemize}
\item Without the map, $N=512$ resolves the solution only for $\delta=10^{-2}$; for $\delta\le10^{-4}$ both conditions fail. The integral form gives the same picture.
\item With the map, only the pinned form reaches the level of rounding errors ($4.4\cdot10^{-15}$--$1.7\cdot10^{-13}$) for every $\delta$, in five Newton iterations, in both formulations.
\item In the integral form the IVP form is up to $4.9$ times more accurate than in the differentiation form, and it converges for $\delta=10^{-6}$. Its errors are still $2.8\cdot10^{3}$ and $2.2\cdot10^{5}$ times the interpolation error for $\delta=10^{-4}$ and $10^{-6}$, and for $\delta\le10^{-8}$ Newton's method does not converge. The amplification of Remark~\ref{rem:sens} is not removed by the reformulation.
\item The integral form is not better conditioned here. After row scaling, the $2$-norm condition numbers of the Newton matrices at the final iterate are $1.5\cdot10^{4}$--$2.2\cdot10^{12}$ for the IVP form and $1.5\cdot10^{4}$--$5.9\cdot10^{6}$ for the pinned form, against $1.7\cdot10^{4}$--$3.8\cdot10^{11}$ and $2.2\cdot10^{2}$--$5.0\cdot10^{2}$ in the differentiation form.
\end{itemize}
Robustness was tested at $N_{\min}$ of Table~\ref{tab:colloc} for $\delta=10^{-6}$ and $10^{-10}$. With the pin at $y_p\in\{0.05,0.1,0.25,0.5,0.75,0.9,0.95\}$ instead of $y^\ast=1/2$, for which Theorem~\ref{cor:propagator}(b) allows an amplification up to $5.3$, the error is $0.94$--$1.17$ times the interpolation error and Newton's method takes five iterations in every case. With the centre of the map shifted by $\pm2\pi\mu$ or $\pm5\pi\mu$, or its width multiplied by $\tfrac14$, $\tfrac12$, $2$ or $4$, Newton's method again takes five iterations and the error stays within $0.57$--$1.23$ times the interpolation error on the perturbed grid. That interpolation error, however, grows to $2\cdot10^{-4}$--$4\cdot10^{-4}$ for a shift of $\pm2\pi\mu$, to $1\cdot10^{-2}$--$2\cdot10^{-2}$ for $\pm5\pi\mu$, and to $2\cdot10^{-5}$--$6\cdot10^{-5}$ and $3\cdot10^{-3}$--$5\cdot10^{-3}$ for the widths $2\pi\mu$ and $4\pi\mu$; the narrower widths cost nothing at these $N$. The pin and the Newton iteration are thus robust, and the accuracy is set by the map, whose centre must be known to within a fraction of a layer width.

\begin{table}[t]
\centering\scriptsize
\setlength{\tabcolsep}{2.5pt}
\caption{Ablation for $q=2$ at $N=512$: maximal error and number of Newton iterations (in parentheses) without and with the sinh map, with the initial condition (IVP) and with the pin, for the differentiation and the integral (Birkhoff) formulation; second column: interpolation error of the exact solution on the sinh-mapped grid ($^\ast$: no convergence in $80$ iterations; diverged: overflow during Newton's method).}
\label{tab:ablation}
\begin{tabular}{cccccccc}
\toprule
 & & \multicolumn{2}{c}{no map, differentiation} & \multicolumn{2}{c}{sinh map, differentiation} & \multicolumn{2}{c}{sinh map, integral}\\
\cmidrule(lr){3-4}\cmidrule(lr){5-6}\cmidrule(lr){7-8}
$\delta$ & interp.\ error & IVP & pin & IVP & pin & IVP & pin\\
\midrule
$10^{-2}$ & $2.0\cdot10^{-15}$ & $2.1\cdot10^{-8}$ (6) & $1.4\cdot10^{-8}$ (5) & $8.3\cdot10^{-14}$ (6) & $5.2\cdot10^{-15}$ (5) & $2.1\cdot10^{-14}$ (6) & $4.4\cdot10^{-15}$ (5) \\
$10^{-4}$ & $1.8\cdot10^{-15}$ & $7.6\cdot10^{-1}$ (7) & $6.1\cdot10^{-1}$ (8) & $2.4\cdot10^{-11}$ (39) & $9.1\cdot10^{-15}$ (5) & $4.9\cdot10^{-12}$ (6) & $4.8\cdot10^{-15}$ (5) \\
$10^{-6}$ & $1.9\cdot10^{-15}$ & $9.5\cdot10^{-1}{}^{\ast}$ (80) & $1.1\cdot10^{0}{}^{\ast}$ (80) & $4.5\cdot10^{-10}{}^{\ast}$ (80) & $4.9\cdot10^{-15}$ (5) & $4.1\cdot10^{-10}$ (26) & $6.3\cdot10^{-15}$ (5) \\
$10^{-8}$ & $2.2\cdot10^{-15}$ & $9.5\cdot10^{-1}{}^{\ast}$ (80) & diverged & $5.7\cdot10^{-8}{}^{\ast}$ (80) & $7.3\cdot10^{-15}$ (5) & $4.4\cdot10^{-8}{}^{\ast}$ (80) & $1.1\cdot10^{-14}$ (5) \\
$10^{-10}$ & $1.7\cdot10^{-15}$ & $9.5\cdot10^{-1}{}^{\ast}$ (80) & diverged & $4.4\cdot10^{-6}{}^{\ast}$ (80) & $5.6\cdot10^{-15}$ (5) & $1.6\cdot10^{-6}{}^{\ast}$ (80) & $1.3\cdot10^{-14}$ (5) \\
$10^{-12}$ & $1.8\cdot10^{-13}$ & $9.5\cdot10^{-1}{}^{\ast}$ (80) & diverged & $1.2\cdot10^{-1}{}^{\ast}$ (80) & $1.7\cdot10^{-13}$ (5) & $1.2\cdot10^{-1}{}^{\ast}$ (80) & $1.7\cdot10^{-13}$ (5) \\
 
\bottomrule
\end{tabular}
\end{table}

\subsection{\texorpdfstring{A stronger degeneracy, $q=3$}{A stronger degeneracy, q=3}}

For $q=3$ the slow phase lasts about $1/(2\delta^2)$ time units and $\mu=2\delta^2$; we take $\delta$ from $10^{-1}$ to $10^{-6}$, i.e.\ $\mu$ down to $2\cdot10^{-12}$, and pin at $y^\ast=2/3$ (Table~\ref{tab:q3colloc}, Fig.~\ref{fig:q3}).
\begin{itemize}
\item For $\delta\le10^{-2}$, the number of nodes needed for $10^{-10}$ agrees with the prediction $\frac{2L}{\pi}\ln10^{10}$ to within the resolution of the $N$-grid.
\item Newton's method converges in five iterations in every case (six for $N=32$ when $\delta\le10^{-5}$). At $N_{\min}$ the initial value is recovered with a relative error of at most $4.6\cdot10^{-12}$.
\item The IVP form fails from $\delta=10^{-4}$ on. At $N=384$ it either does not converge or converges to wrong discrete solutions with errors between $0.2$ and $0.4$. Already for $\delta=10^{-3}$ it converges to wrong solutions (errors $0.22$ and $0.39$ for $N=128$ and $192$) or does not converge ($N=160$, $224$--$352$ and $416$--$512$). Its Newton matrices have condition numbers up to $3.5\cdot10^{18}$. The condition numbers of the pinned form grow like $\mu^{-1}$, from $10^{2}$--$2.5\cdot10^{4}$ at $\delta=10^{-1}$ to $2\cdot10^{11}$--$10^{12}$ at $\delta=10^{-6}$.
\end{itemize}
Fig.~\ref{fig:grid}(b) collects the node counts of both families as functions of the layer width. The counts grow by $32$ nodes per decade of $\mu$ for both $q=2$ and $q=3$ (by $48$ between $\mu=10^{-2}$ and $10^{-4}$ for $q=2$), close to the slope $\frac2\pi\ln10\cdot\ln10^{10}\approx33.8$ predicted by \eqref{eq:rate}. For $q=2$ the constant $C$ in \eqref{eq:rate} is slightly larger, which shifts the counts by about one step of the $N$-grid.

\begin{table}[t]
\centering\small
\caption{$q=3$: sinh-mapped Chebyshev collocation for $\mu y'=y^3(1-y)$, $y(-1)=\delta$, $\mu=2\delta^2$; columns as in Table~\ref{tab:colloc}, pin $y(x^\ast)=2/3$.}
\label{tab:q3colloc}
\begin{tabular}{ccccccccc}
\toprule
 & & & \multicolumn{4}{c}{pinned form} & \multicolumn{2}{c}{IVP form, $N=384$}\\
\cmidrule(lr){4-7}\cmidrule(lr){8-9}
$\delta$ & $L$ & $\frac{2L}{\pi}\ln10^{10}$ & $N_{\min}$ & error & its & $|y(-1)/\delta-1|$ & error & its\\
\midrule
$10^{-1}$ & 3.43 & 50 & 64 & $1.7\cdot10^{-13}$ & 5 & $3.3\cdot10^{-14}$ & $2.0\cdot10^{-14}$ & 7 \\
$10^{-2}$ & 8.07 & 118 & 128 & $6.1\cdot10^{-12}$ & 5 & $3.7\cdot10^{-13}$ & $1.3\cdot10^{-12}$ & 52 \\
$10^{-3}$ & 12.67 & 186 & 192 & $1.6\cdot10^{-11}$ & 5 & $1.1\cdot10^{-12}$ & $4.6\cdot10^{-10}$ & 73 \\
$10^{-4}$ & 17.28 & 253 & 256 & $2.7\cdot10^{-11}$ & 5 & $2.7\cdot10^{-12}$ & $3.7\cdot10^{-1}{}^{\ast}$ & 80 \\
$10^{-5}$ & 21.88 & 321 & 320 & $3.7\cdot10^{-11}$ & 5 & $3.9\cdot10^{-12}$ & $3.0\cdot10^{-1}{}^{\ast}$ & 80 \\
$10^{-6}$ & 26.49 & 388 & 384 & $4.5\cdot10^{-11}$ & 5 & $4.6\cdot10^{-12}$ & $2.3\cdot10^{-1}$ & 77 \\
 
\bottomrule
\end{tabular}
\end{table}

\begin{figure}[t]
\centering
\includegraphics[width=0.62\textwidth]{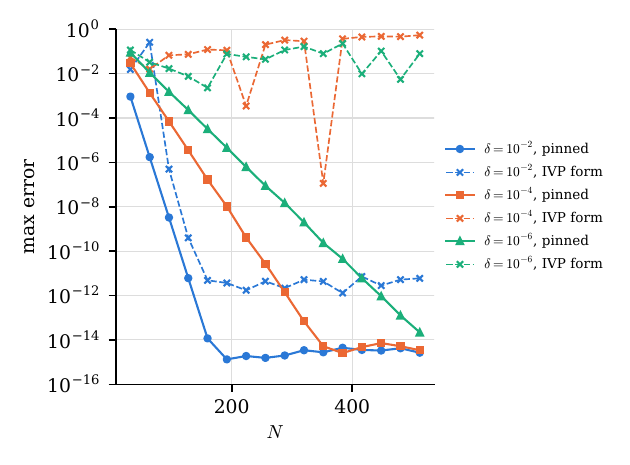}
\caption{$q=3$: maximal error versus $N$ of the pinned form (solid) and of the IVP form (dashed).}
\label{fig:q3}
\end{figure}

\subsection{Comparison with Chebfun}\label{sec:chebfun}

Chebfun \cite{ChebfunGuide} offers three ways to solve the flame problem $y'=y^2-y^3$, $y(0)=\delta$, on $[0,2/\delta]$ without any rescaling: its default IVP solver, which integrates with \texttt{ode113} and converts the result to a chebfun; global spectral collocation with the initial condition (\texttt{solvebvp}, the default before version~5.1); and the same collocation with the initial condition replaced by the interior condition $y(t^\ast)=2/3$, $t^\ast=T_\delta(2/3)$, imposed as a functional boundary condition. The value $2/3$ is the optimal pin for the variable $y$ (Section~\ref{sec:pin}); there is no coordinate map. We ran all three with Chebfun version 5.7.0 \cite{ChebfunGuide} and its default settings in MATLAB R2023b (Table~\ref{tab:chebfun}).
\begin{itemize}
\item Collocation with the initial condition converges only for $\delta=10^{-1}$. For $\delta\le10^{-2}$ Newton's method fails after $24$--$26$ iterations.
\item Collocation with the interior condition converges for $\delta=10^{-2}$ as well, in $13$ iterations with $879$ Chebyshev coefficients and error $9\cdot10^{-12}$, where the initial condition fails. This confirms the effect of the pin independently of our map and implementation. For $\delta\le10^{-3}$ it fails in the first iterations: without a map the solution needs about $23/(\pi\delta)\approx7300$ points for $\delta=10^{-3}$ (Section~\ref{sec:approx}), and Chebfun's default initial guess does not resemble the front.
\item The default \texttt{ode113} solver succeeds, but its error grows from $1.4\cdot10^{-10}$ for $\delta=10^{-1}$ to $3.3\cdot10^{-4}$ for $\delta=10^{-4}$, as the position of the front becomes more sensitive to the errors made in the slow phase (Remark~\ref{rem:sens}). For $\delta\le10^{-5}$ the computed solution cannot be represented with $65537$ points and the error is about $0.9$.
\item The method of this paper needs $33$--$225$ points and five Newton iterations for an error below $10^{-10}$ at every $\delta$.
\end{itemize}

\begin{table}[t]
\centering\small
\setlength{\tabcolsep}{3.5pt}
\caption{Chebfun 5.7.0 with default settings on $y'=y^2-y^3$, $y(0)=\delta$, $t\in[0,2/\delta]$. Length: number of Chebyshev coefficients (at termination for failed runs); its: Newton iterations; failed: Newton's method did not converge; not resolved: the \texttt{ode113} solution could not be represented with $65537$ points. Last three columns: the method of this paper (sinh map, pin), smallest $N$ for error $\le10^{-10}$, number of points $N+1$.}
\label{tab:chebfun}
\begin{tabular}{cccccccccccc}
\toprule
 & \multicolumn{2}{c}{\texttt{ode113}} & \multicolumn{3}{c}{collocation, $y(0)=\delta$} & \multicolumn{3}{c}{collocation, $y(t^\ast)=2/3$} & \multicolumn{3}{c}{this paper}\\
\cmidrule(lr){2-3}\cmidrule(lr){4-6}\cmidrule(lr){7-9}\cmidrule(lr){10-12}
$\delta$ & length & error & length & its & error & length & its & error & $N+1$ & its & error\\
\midrule
$10^{-1}$ & 82 & $1.4\cdot10^{-10}$ & 99 & 8 & $3.5\cdot10^{-13}$ & 94 & 7 & $5.1\cdot10^{-13}$ & 33 & 5 & $1.4\cdot10^{-11}$ \\
$10^{-2}$ & 721 & $7.3\cdot10^{-8}$ & 1377 & 26 & failed & 879 & 13 & $9.0\cdot10^{-12}$ & 65 & 5 & $6.8\cdot10^{-11}$ \\
$10^{-3}$ & 6564 & $4.2\cdot10^{-6}$ & 3170 & 26 & failed & -- & 1 & failed & 129 & 5 & $4.9\cdot10^{-14}$ \\
$10^{-4}$ & 60313 & $3.3\cdot10^{-4}$ & 2667 & 24 & failed & -- & 1 & failed & 161 & 5 & $4.7\cdot10^{-13}$ \\
$10^{-5}$ & -- & not resolved & 2659 & 26 & failed & -- & 2 & failed & 193 & 5 & $1.9\cdot10^{-12}$ \\
$10^{-6}$ & -- & not resolved & 2921 & 26 & failed & -- & 5 & failed & 225 & 5 & $4.6\cdot10^{-12}$ \\
 
\bottomrule
\end{tabular}
\end{table}

\subsection{Extended precision}\label{sec:hp}

In double precision the errors of Tables~\ref{tab:colloc} and~\ref{tab:q3colloc} stop at the level of rounding errors, which hides both the asymptotic convergence rate and the cause of the failure of the IVP form. We therefore repeated the computations in $256$-bit arithmetic (about $77$ digits), with the transcendental functions from mpmath \cite{mpmath} and the linear systems solved by LU factorization in FLINT/Arb \cite{Johansson2017}, for $q=2$, $\delta\in\{10^{-2},10^{-4},10^{-6}\}$ and $q=3$, $\delta\in\{10^{-1},10^{-2},10^{-3}\}$, $N=64,128,\dots,512$, with the same initial guess and Newton's method stopped as in step~6 of the algorithm with the tolerances $2^{-236}$ and $2^{-226}$ (Fig.~\ref{fig:hp}). Errors are measured on $600$ Chebyshev points of the first kind against the exact solution computed to working precision.
\begin{itemize}
\item The pinned form converges in $7$ ($q=2$) or $8$ ($q=3$) Newton iterations in every case. Its error is between $0.91$ and $1.00$ times the interpolation error of the exact solution on the same grid, down to the working precision $\approx10^{-76}$. The $256$-bit reference values agree with values recomputed at $512$ bits to $1\cdot10^{-76}$, so errors near the working precision are at the level of the reference.
\item The convergence rate is the predicted one. A least-squares fit of $\ln(\text{error})$ against $N$ (errors above $10^{-70}$) gives slopes between $0.999$ and $1.011$ times the exact exponent $\ln\rho_\ast$ of Section~\ref{sec:approx} for all six cases, and between $0.976$ and $1.008$ times its approximation $\pi/(2L)$. For $q=3$, where the identification of $\pm\pi i$ as the nearest singularities rests on Proposition~\ref{prop:sing}, this is an independent confirmation.
\item The failure of the IVP form is not a rounding effect. For $q=2$ it converges in all cases but one ($\delta=10^{-6}$, $N=128$), but its errors exceed the interpolation error by factors up to $190$ for $\delta=10^{-4}$ and up to $7.8\cdot10^{3}$ for $\delta=10^{-6}$: the discretization error near $x=-1$ is amplified as described by Proposition~\ref{prop:pin}. For $q=3$ and $\delta\le10^{-2}$, Newton's method from the generic initial guess either does not converge in $40$ iterations or converges to wrong discrete solutions (errors $0.2$--$0.4$ for $\delta=10^{-3}$, $N=128,192$), although the pinned form converges in eight iterations from the same guess. Started from the converged pinned solution, the IVP form does converge for $\delta=10^{-2}$, but with errors $240$ and $63$ times the interpolation error at $N=64$ and $128$.
\end{itemize}

\begin{figure}[t]
\centering
\includegraphics[width=\textwidth]{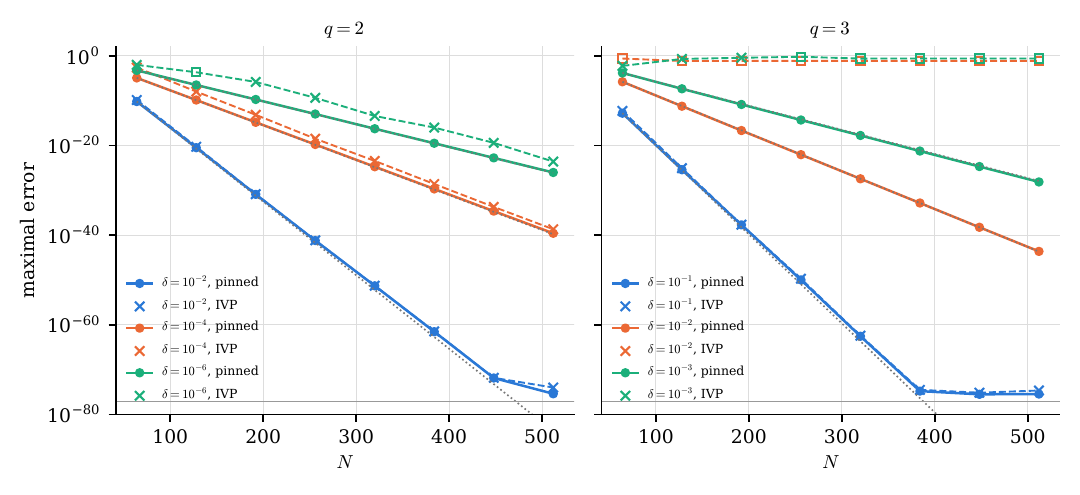}
\caption{Maximal error versus $N$ in $256$-bit arithmetic. Solid: pinned form; crosses and dashed: IVP form (open squares: Newton did not converge in $40$ iterations); dotted: the rate $e^{-\pi N/(2L)}$ of \eqref{eq:rate}; horizontal line: $2^{-256}$. Errors below about $10^{-75}$ are at the accuracy of the reference solution.}
\label{fig:hp}
\end{figure}

\subsection{An Arrhenius reaction rate}\label{sec:general}

The method uses three quantities: the time function, which gives the scale $\mu$ and the position of the condition; the nearest complex singularity of its inverse, which gives the centre and width of the sinh map; and the maximiser of $|h|$ in Theorem~\ref{cor:propagator}, which gives the condition point. None of them requires a closed form. We demonstrate this on the Arrhenius reaction rate of combustion theory in the Zeldovich--Frank-Kamenetskii form \cite{BuckmasterLudford1982,Zeldovich1985},
\[
y'=f(y):=(1-y)\exp\Bigl(-\frac{\beta(1-y)}{1-\alpha(1-y)}\Bigr),\qquad y(0)=0,
\]
which describes the adiabatic thermal explosion of a homogeneous mixture with reactant consumption. Here $y$ is the reduced temperature, $\beta$ the Zeldovich number and $\alpha$ the heat release parameter, and the rate constant is absorbed into the time unit. We take $\alpha=0.8$ and $\beta=2,\dots,6$. Three features distinguish this problem from \eqref{eq:fq}.
\begin{itemize}
\item The initial state is not an equilibrium: $f(0)=e^{-\beta/(1-\alpha)}>0$. The long induction phase is caused by this small rate and not by a degenerate equilibrium, and the initial condition $y(0)=0$ is the natural one, with no artificial $\delta$. The real orbit leaves $y_e=1-1/\alpha<0$, where $f$ vanishes to infinite order from the right, at $t=-\infty$ and tends to $y=1$.
\item For $\alpha>0$ the time function has no closed form. For $\alpha=0$, the Frank-Kamenetskii approximation, $\Theta(y)=\mathrm{Ei}(\beta(1-y_f))-\mathrm{Ei}(\beta(1-y))$ with the exponential integral $\mathrm{Ei}$, which we use only as a check.
\item The singularity does not come from $y=\infty$. Since $1/f(z)=-e^{-\beta/\alpha}/z+O(z^{-2})$ as $|z|\to\infty$, the time function grows logarithmically there, and the hypothesis $1/f=o(z^{-1})$ of Remark~\ref{rem:residue} fails. Instead, $f$ has an essential singularity at $y_e$, and $1/f(z)$ behaves like $\exp\bigl(\beta/(\alpha^2(z-y_e))\bigr)$ near $y_e$, so it vanishes to infinite order as $z\to y_e$ from the left. The candidate singularity of $Y=\Theta^{-1}$ is the corresponding asymptotic value,
\[
s_\ast=\int_{y_f}^{y_e}\frac{\dd z}{f(z)},
\]
taken along a path through the upper half-plane that reaches $y_e$ from the left. This path and its mirror image differ by a loop around $y_e$, so $\Im s_\ast=\pi\operatorname{Res}_{z=y_e}(1/f)$. The residues of $1/f$ at $y=1$, at $y_e$ and at infinity add up to zero, and the first and the last are $-1$ and $e^{-\beta/\alpha}$. Hence
\[
\Im s_\ast=\pi\bigl(1-e^{-\beta/\alpha}\bigr),
\]
slightly less than the value $\pi/|f'(1)|=\pi$ of Remark~\ref{rem:residue}. Unlike for $f_q$, we have no proof that no other singular value lies closer to the real axis; the AAA test below and the observed convergence rates support it.
\end{itemize}
We compute everything numerically, as one would for a general $f$.
\begin{itemize}
\item The time function, normalised by $\Theta(y_f)=0$, is written as $-\ln\frac{1-y}{1-y_f}+\int_{y_f}^y\bigl(\frac1{f(z)}-\frac1{1-z}\bigr)\dd z$, whose integrand is smooth on $[0,1]$, and the integral is evaluated by composite Gauss--Legendre quadrature ($4000$ panels of $14$ points). The solution $Y$ is obtained from it by Newton's method; this is also the reference solution. At ten points spread over the interval it agrees with $40$-digit quadrature to $2\cdot10^{-16}$.
\item The singularity is computed by quadrature in $30$-digit arithmetic \cite{mpmath} along the polygon $y_f\to y_f+40i\to y_e-1+40i\to y_e-1\to y_e$. Its imaginary part agrees with $\pi(1-e^{-\beta/\alpha})$ to $4\cdot10^{-16}$ (Table~\ref{tab:arrhenius}). For $\alpha=0$ the same quadrature, along a path to $+\infty$, reproduces $\mathrm{Ei}(\beta(1-y_f))+\pi i$ to all digits.
\item Since $y(0)=0$, the unknown is $y$ itself, and by Theorem~\ref{cor:propagator} the condition is imposed at the maximiser $y_f$ of $f$. With $T=\Theta(y_f)-\Theta(0)$, $\mu=1/T$ and $x=\mu t-1$, the condition point is $x^\ast=0$. The IVP form amplifies errors by $f(y_f)/f(0)$, which ranges from $3.0\cdot10^{3}$ for $\beta=2$ to $5.8\cdot10^{11}$ for $\beta=6$.
\end{itemize}
The sinh map is centred at $\mu\Re s_\ast$ with $\eps=\mu\Im s_\ast$, and Newton's method starts from the generic profile of Section~\ref{sec:pin}, which equals $y_f$ at the pin. The results are shown in Table~\ref{tab:arrhenius} and Fig.~\ref{fig:general}.
\begin{itemize}
\item For $N\ge64$ the pinned form converges in all cases, in $8$, $12$, $16$, $20$--$30$ and $24$--$28$ Newton iterations for $\beta=2,\dots,6$; for $N=32$ and $\beta=5,6$, where the interpolation error is about $3\cdot10^{-2}$, it does not converge. The iteration counts are larger than in Section~\ref{sec:q2} and grow with $\beta$, since the generic initial guess was built for \eqref{eq:fq}.
\item The error is $0.57$--$3.1$ times the interpolation error of the reference solution on the same grid, and it reaches $1.6\cdot10^{-15}$--$3.4\cdot10^{-15}$ for every $\beta$. The observed convergence rates are $0.96$--$1.05$ times $\ln\rho_\ast$. The number of nodes for accuracy $10^{-10}$ grows by $64$ per unit of $\beta$, i.e.\ by about $32$ per decade of $\mu$, as for \eqref{eq:fq} (Fig.~\ref{fig:grid}(b)).
\item The IVP form has errors between $0.09$ and $4.9$ at every $N$ for $\beta\ge3$. For $\beta=2$, where $\mu=1.8\cdot10^{-3}$, it fails for $N\le128$ and reaches $2\cdot10^{-14}$--$2\cdot10^{-12}$ for $N\ge160$.
\end{itemize}
AAA approximation of the computed solution for $\beta=6$ locates the singularity as accurately as for \eqref{eq:fq} (Table~\ref{tab:aaa}).

\begin{table}[t]
\centering\footnotesize
\setlength{\tabcolsep}{3.5pt}
\caption{Arrhenius rate, $\alpha=0.8$: rate $f(0)$ at the initial state, layer width $\mu$, condition point $y_f$, singularity $s_\ast$ (in the variable $s=x/\mu$, with $s=0$ at the condition point) and map parameter $L$, all computed numerically; smallest $N\in\{32,64,\dots,512\}$ for which the pinned form has maximal error $\le10^{-10}$, with its error and number of Newton iterations; error of the IVP form at the same $N$ ($^\ast$: Newton's method did not converge in $80$ iterations).}
\label{tab:arrhenius}
\begin{tabular}{cccccccccc}
\toprule
$\beta$ & $f(0)$ & $\mu$ & $y_f$ & $s_\ast$ & $L$ & $N_{\min}$ & pinned & its & IVP\\
\midrule
2 & $4.5\cdot10^{-5}$ & $1.8\cdot10^{-3}$ & 0.7070 & $0.2759+2.8837i$ & 6.0 & 96 & $9.9\cdot10^{-12}$ & 8 & $6.8\cdot10^{-1}$ \\
3 & $3.1\cdot10^{-7}$ & $2.0\cdot10^{-5}$ & 0.7756 & $0.7132+3.0677i$ & 10.4 & 160 & $4.5\cdot10^{-11}$ & 12 & $5.2\cdot10^{-1}$ \\
4 & $2.1\cdot10^{-9}$ & $1.9\cdot10^{-7}$ & 0.8176 & $0.9799+3.1204i$ & 15.0 & 224 & $4.5\cdot10^{-11}$ & 16 & $4.5\cdot10^{-1}$ \\
5 & $1.4\cdot10^{-11}$ & $1.6\cdot10^{-9}$ & 0.8462 & $1.1495+3.1355i$ & 19.8 & 288 & $3.5\cdot10^{-11}$ & 20 & $5.0\cdot10^{-1}$ \\
6 & $9.4\cdot10^{-14}$ & $1.3\cdot10^{-11}$ & 0.8669 & $1.2647+3.1399i$ & 24.6 & 352 & $2.3\cdot10^{-11}$ & 24 & $4.7\cdot10^{-1}$ \\
 
\bottomrule
\end{tabular}
\end{table}

\begin{figure}[t]
\centering
\includegraphics[width=\textwidth]{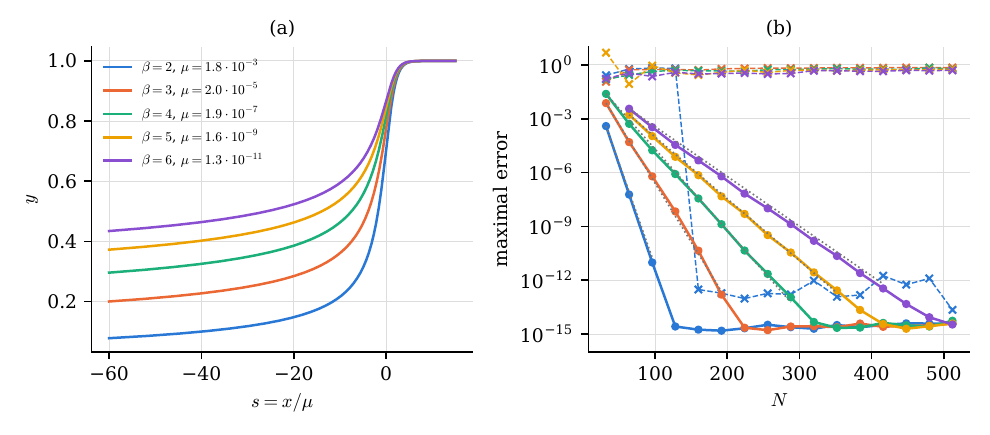}
\caption{Arrhenius rate, $\alpha=0.8$, $\beta=2,\dots,6$, with the time function, the singularity and the condition point computed numerically. (a) Solution in the variable $s=x/\mu$ ($s=0$ at the condition point). (b) Maximal error versus $N$ of the pinned form (solid) and of the IVP form (dashed, crosses); dotted: the rate $e^{-\pi N/(2L)}$.}
\label{fig:general}
\end{figure}

\subsection{Locating the singularities from the computed solution}\label{sec:aaa}

The map parameters above were taken from Proposition~\ref{prop:sing} and Remark~\ref{rem:residue}. The singularities can also be located a posteriori from a computed solution, by rational approximation with the AAA algorithm \cite{NakatsukasaSeteTrefethen2018}, a standard tool for numerical analytic continuation \cite{Trefethen2023}; see \cite{Gheorghiu2026} for an application to Laguerre collocation solutions of a boundary value problem on the half-line. We evaluate the pinned collocation solution at $3000$ equispaced points of $s\in[-30,30]$ ($x\in[-1,1]$ for the unmapped grid) and compute its AAA approximant with relative tolerance $10^{-13}$, discarding poles with residues below $10^{-8}$, which we treat as small-residue artifacts (Froissart doublets \cite{NakatsukasaSeteTrefethen2018}). A branch point appears in the approximant as a string of poles along a branch cut \cite{Trefethen2023}. Table~\ref{tab:aaa} lists the poles nearest to the real axis.
\begin{itemize}
\item In all cases the lowest pole lies within $2\%$ above the exact branch point $s_\ast$, with real part correct to $4\cdot10^{-3}$, and the next poles are stacked vertically above it.
\item The same pattern is obtained from the exact solution sampled at the same points (last column). The small offset is therefore a property of the rational representation of a branch point, not an error of the collocation solution.
\item The first row uses the unmapped Chebyshev solution for $\delta=10^{-2}$, computed on a grid that uses no information about the singularities. It locates them as accurately as the mapped solutions.
\end{itemize}
For problems without a computable time function, a coarse solution followed by AAA could thus provide the centre and width of the map, as in the a posteriori approach of \cite{TeeTrefethen2006}; we have not tested such a two-stage procedure here.

\begin{table}[t]
\centering\footnotesize
\setlength{\tabcolsep}{3pt}
\caption{Poles of the AAA approximants of the computed pinned solutions in the upper half-plane, in the variable $s=(x-x_0)/\mu$ ($s=x/\mu$ for the Arrhenius rate of Section~\ref{sec:general}): the lowest pole and the imaginary parts of the next two. Last column: lowest pole of the AAA approximant of the exact solution (for the Arrhenius rate the high-precision reference solution) sampled at the same points.}
\label{tab:aaa}
\begin{tabular}{ccccccc}
\toprule
problem & $N$ & grid & exact $s_\ast$ & lowest pole & next poles, $\Im$ & exact solution\\
\midrule
$q=2$, $\delta=10^{-2}$ & 1024 & unmapped & $3.1416i$ & $-0.0014+3.1772i$ & 3.468, 4.079 & $-0.0036+3.1765i$ \\
$q=2$, $\delta=10^{-6}$ & 256 & sinh map & $3.1416i$ & $-0.0034+3.1822i$ & 3.513, 4.207 & $-0.0010+3.1777i$ \\
$q=2$, $\delta=10^{-10}$ & 384 & sinh map & $3.1416i$ & $-0.0014+3.1888i$ & 3.560, 4.202 & $-0.0010+3.1777i$ \\
$q=3$, $\delta=10^{-3}$ & 256 & sinh map & $3.1416i$ & $-0.0020+3.1899i$ & 3.494, 4.110 & $-0.0024+3.1885i$ \\
$q=3$, $\delta=10^{-6}$ & 416 & sinh map & $3.1416i$ & $-0.0038+3.1988i$ & 3.517, 4.042 & $-0.0024+3.1885i$ \\
Arrhenius, $\beta=6$ & 384 & sinh map & $1.2647+3.1399i$ & $1.2668+3.2014i$ & 3.580, 4.203 & $1.2685+3.1904i$ \\
 
\bottomrule
\end{tabular}
\end{table}

\section{Conclusions}

Global spectral collocation for the ignition problems $y'=y^q(1-y)$, $y(0)=\delta$, faces two obstacles: an interior layer that polynomials resolve only with $O(1/\mu)$ nodes, and an initial condition that makes the discrete problem ill-conditioned. Both can be overcome by exploiting the structure of the equation.
\begin{itemize}
\item All solutions are translates of the master profile $Y_q=\Theta_q^{-1}$, whose singularities lie exactly at $\pm(2k+1)\pi i$. The position of the transition layer, $x_0$, and the half-width $\pi\mu$ of the strip of analyticity are therefore known in closed form. The value $\pi=\pi/|f_q'(1)|$ is set by the relaxation rate at the stable equilibrium.
\item For the sinh map of Tee and Trefethen built on these two parameters, the analyticity of the solution predicts the rate $\rho_\ast^{-N}\approx e^{-\pi N/(2L)}$, $L\approx\ln(2/(\pi\mu))$, so that $O(\ln(1/\mu))$ nodes suffice instead of $O(1/\mu)$; this rate is observed in double and in $256$-bit arithmetic. Centring the map only a few layer widths away from $x_0$ loses most of this gain.
\item Replacing the initial condition by an equivalent interior condition at the maximiser of $|\varphi'f_q|$, for the unknown $u=\varphi(y)$, bounds the propagator of the continuous linearized equation by one; for $u=\ln y$ this is the maximiser of $f_q(y)/y$. In the experiments Newton's method then converges in five or six iterations to the interpolation error for layer widths down to $10^{-12}$, also with the pin far from its optimal position, whereas the IVP form fails or converges to wrong discrete solutions, also in an integral (Birkhoff) reformulation.
\end{itemize}
The construction requires only the time function $T(y)=\int\dd y/f(y)$ and the nearest singularity of its inverse. For scalar autonomous equations $y'=f(y)$ on a monotone orbit, $T$ can be evaluated by quadrature, and a candidate singularity follows from Remark~\ref{rem:residue}, from quadrature along a path, or a posteriori from a computed solution by AAA approximation; that it is the nearest one needs an argument as in Proposition~\ref{prop:sing} or numerical evidence. For an Arrhenius reaction rate, whose initial state is not an equilibrium and whose singularity comes from an essential singularity of $f$, Section~\ref{sec:general} shows in five cases that the numerically computed quantities give the same accuracy and convergence rates as the closed forms do for \eqref{eq:fq}, although Newton's method needs more iterations. Problems without translation invariance, such as non-autonomous equations or systems, are a different matter. There the map parameters can still be estimated a posteriori (Section~\ref{sec:aaa}), but the position of an equivalent interior condition is not known in advance; extending the pinning idea to such problems is left for future work.

\section*{Acknowledgements}
The author thanks C\u{a}lin-Ioan Gheorghiu for his careful reading of earlier versions of this paper and for many valuable comments and suggestions, in particular on the qualitative analysis of the equilibria, the comparison with Chebfun, and the use of the AAA algorithm.

\section*{Declarations}

\paragraph{Funding.} This research received no specific grant from any funding agency in the public, commercial, or not-for-profit sectors.

\paragraph{Competing interests.} The author has no competing interests to declare.

\paragraph{Data and code availability.} All scripts and data needed to reproduce the tables and figures are publicly available at \url{https://github.com/boros91/flame-collocation} and archived at Zenodo, \url{https://doi.org/10.5281/zenodo.23479185}. They include implementations of the method in Python and MATLAB, the double-precision experiments (Python), the extended-precision computations of Section~\ref{sec:hp} in Python (mpmath and FLINT/Arb), the computations of Section~\ref{sec:general}, and the MATLAB script for the Chebfun comparison of Section~\ref{sec:chebfun}.

\paragraph{Author contributions.} The author is the sole contributor to this work.

\paragraph{Use of AI tools.} The author used the AI assistant Claude (Anthropic) to help write and check the code, verify the numerical results against the saved outputs, and edit the text. All results were checked by the author, who takes full responsibility for the content of this paper.


\end{document}